\documentclass[12pt, reqno]{amsart}
\usepackage[margin=1in]{geometry}

\usepackage{amssymb}
\usepackage{amsthm}

\usepackage{color} 

\usepackage{hyperref,amsmath,caption,amscd}

\newtheorem{theorem}{Theorem}[section]
\newtheorem{lemma}[theorem]{Lemma}
\newtheorem{prop}[theorem]{Proposition}
\newtheorem{cor}[theorem]{Corollary}

\newtheorem{predefinition}[theorem]{Definition}
\newenvironment{defn}{\begin{predefinition}\rm}{\end{predefinition}}
\newtheorem{preremark}[theorem]{Remark}
\newenvironment{remark}{\begin{preremark}\rm}{\end{preremark}}
\newtheorem{prenotation}[theorem]{Notation}

\newtheorem{preexample}[theorem]{Example}
\newenvironment{example}{\begin{preexample}\rm}{\end{preexample}}
\newtheorem{preclaim}[theorem]{Claim}

\newtheorem{prequestion}[theorem]{Question}

\makeatletter

\newcommand{\FF}{{\mathbb F}}
    
\newcommand{\cent}{\rm{Cent}_{{\rm Aut}_{k}(C)}\langle \iota \rangle}

\title{A Mass Formula For Artin--Schreier Curves Over Finite Fields}

\author{Anne M. Ho}
\email{aho5@utk.edu}
\address{University of Tennessee at Knoxville, Department of Mathematics, 
227 Ayres Hall, 1403 Circle Drive, Knoxville, TN 37996-1320}

\author{Rachel Pries}
\email{pries@colostate.edu}
\address{Colorado State University, Department of Mathematics, Weber 101, Fort Collins, CO 80523-1874}

\begin{document}

\maketitle

\begin{abstract}
We study a mass formula for Artin--Schreier curves of genus $g$ 
defined over a finite field $k$ of characteristic $p$.
For an odd prime $p$ and for small $g$, we determine the number 
of $k$-isomorphism classes of Artin--Schreier curves of genus $g$, 
weighted by the order of the centralizer of the 
Artin--Schreier involution in the automorphism group.
This extends earlier results by several authors in characteristic $p=2$. 
MSC10: 11G20, 11T06, 12F10, 14G15, 14H10, 14H37, 62R01. 
\end{abstract}


\section{Introduction}\label{intro}

Lang and Weil \cite{langweil} showed that the number of points of a variety over a finite field
gives a lot of information about its geometry.
For a prime $p$, let $k = {\mathbb F}_q$ be a finite field whose cardinality $q$ is a power of $p$. 
If $V$ is an irreducible projective algebraic variety of dimension $d$ over $k$, 
then $q^d$ is the dominant term for the number of $k$-points of $V$. 
Recently, people developed an interest in the number of $k$-points on 
moduli spaces of curves, because this yields information about the cohomology of these moduli spaces;
see, e.g.\ \cite{VDGsurvey} for a survey of many results in this area.

Suppose $C$ is a curve over $k$.
We assume throughout this paper that 
$C$ is smooth, projective, and geometrically irreducible.
Let $g$ be the genus of $C$ and suppose that $g \geq 1$.

For small $g$, the 
number of isomorphism classes $[C]$ of curves $C$ of genus $g$ over $k$ is known.
This number is given by a mass formula, in which the contribution from $[C]$ is weighted by 
$|{\rm Aut}_k(C)|^{-1}$, 
where ${\rm Aut}_k(C)$ denotes the automorphism group of $C$ over $k$.

When $g=1$, an elliptic curve $E$ over $k$ defines a $k$-point of the moduli space 
$\mathcal M_{1,1}$.
For every prime power $q$, Howe proved that $\sum_{[E] \in {\mathcal M}_{1,1}(k)/\simeq_k} |{\rm Aut}_k(E)|^{-1}= q$ \cite[Corollary 2.2]{Howe}.

This was generalized for hyperelliptic curves of arbitrary genus.
A hyperelliptic curve of genus $g$ over $k$ defines a $k$-point of the moduli space
${\mathcal H}_g$.  As in \cite[Definition~2.2]{bergstrom}, let:
\begin{equation} \label{Defa0}
a({\mathcal H}_g)_0\mid_g = \sum_{[C] \in {\mathcal H}_g(k)/\simeq_k} |{\rm Aut}_k(C)|^{-1}.
\end{equation}
Brock and Granville proved that $a({\mathcal H}_g)_0\mid_g=q^{2g-1}$ 
when $q$ is odd, \cite[Proposition~7.1]{BrockGranville}.
Bergstr\"om proved the same formula is true when $q$ is even, 
\cite[Theorem~10.3]{bergstrom}.

For $q$ even and $g=2,3$, the fact that $a({\mathcal H}_g)_0\mid_g=q^{2g-1}$ was proven earlier: 
when $g=2$ by \cite[Theorem 18]{Nart2}, building on \cite[Corollary~5.3]{VanDerGeer};
and when $g=3$ by \cite[Theorem~8]{Nart3}.  

In characteristic $p=2$, hyperelliptic curves are Artin--Schreier curves by definition.
A curve $C$ is an {\it Artin--Schreier curve} if there is a Galois cover $\pi: C \rightarrow \mathbb{P}^1$ of degree $p$.
In this paper, we extend the results about hyperelliptic curves in characteristic $2$ 
to the case of Artin--Schreier curves of low genus in arbitrary characteristic $p$.
To do this, we consider the moduli space ${\mathcal AS}_g$ 
of Artin--Schreier curves of genus $g$ as in \cite[Proposition 2.7]{maugeais}.
  
We now state the results more precisely.  An Artin--Schreier curve $C$ over $k$ 
has an equation of the form $y^p-y=u(x)$ for some rational function $u(x) \in k(x)$.
If $C$ is an Artin--Schreier curve,
then the genus of $C$ is a multiple of $(p-1)/2$ by the wild Riemann-Hurwitz formula.  
Let $\iota$ denote the order $p$ automorphism 
$\iota := (x,y) \mapsto (x,y+1)$, which generates the Galois group of $\pi$.
Let $\cent$ be the centralizer of $\langle \iota \rangle$ in ${\rm Aut}_k(C)$.
The reason for introducing the centralizer is that $\langle \iota \rangle$ may not be in the center of ${\rm Aut}_k(C)$ when $p$ is odd.

We consider the mass formula: 
\begin{equation} \label{Eweightsum}
Z_g(q) := a({\mathcal AS}_g)_0\mid_g := \sum_{[C] \in {\mathcal AS}_g(k)/\simeq_k} |\cent|^{-1},
\end{equation}
where $[C]=(C, \iota)$ ranges over the $k$-isomorphism classes of Artin--Schreier curves 
with automorphism such that $C$ has genus $g$.
Note that \eqref{Eweightsum} specializes to \eqref{Defa0} when $p=2$ since $\iota$ is in the center 
of ${\rm Aut}_k(C)$ in that case.

The main result of this paper is Theorem~\ref{resultsforpnot3} (see also 
Theorem \ref{results for p=3}).
In it, we determine the mass formula $Z_g(q)$ for an arbitrary odd prime power $q$ when the genus is $g = d(p-1)/2$ for $1 \leq d \leq 5$. 
For $p \geq 7$, the result states that 
	\[Z_g(q) =
		\begin{cases}
		1 & \rm{if }\  g = 1(p-1)/2,\\
		2q-1 & \rm{if }\  g = 2(p-1)/2,\\
		2q^2-q & \rm{if }\  g = 3(p-1)/2,\\
		4q^3-3q^2 & \rm{if }\  g = 4(p-1)/2, \\ 
		4q^4-4q^3+q^2 & \rm{if }\  g = 5(p-1)/2. \\ 
		\end{cases}\]

These formulas have geometric significance, as explained in Section~\ref{dimension}. 
The degree of $Z_g(q)$ as a polynomial in $q$ is the dimension of ${\mathcal AS}_g$; 
if $g=d(p-1)/2$, every irreducible component of ${\mathcal AS}_g$ has dimension $d-1$ \cite[Corollary 3.16]{maugeais}.
The leading coefficient of $Z_g(q)$ is the number of irreducible components of ${\mathcal AS}_g$;
this number is determined in \cite{Pries}.
One interesting fact is that the number of irreducible components 
is sometimes different for small primes $p$ and this is reflected in the formulas we find.
The lower order terms of our formulas provide new information about the
cohomology of ${\mathcal AS}_g$.

The proof of Theorem~\ref{resultsforpnot3} builds on Theorem~\ref{weightedsum}, 
in which we re-express $Z_g(q)$
as an average over conjugacy classes in a symmetric group.

To obtain the formulas, we use similar methods as in \cite{Nart2} and \cite{Nart3}.
We separate into cases indexed by discrete information for the cover $\pi: C \to {\mathbb P}^1$, namely the ramification data 
and splitting behavior.  
In each case, we study the orbits and stabilizers for the 
action of $\rm{PGL}_2(k)$ on the Artin--Schreier equations.
In Proposition~\ref{Results123}, we determine the mass formula for the set of Artin--Schreier curves
of arbitrary genus $g$ over $k$
when the number of branch points of $\pi$ is at most $3$.
This allows us to finish the proof of Theorem~\ref{resultsforpnot3}.

We also have some results when $\pi$ has $4$ branch points.
This case is more difficult because it is more complicated to study the 
group action on covers branched at $4$-sets of 
$\mathbb{P}^1(\bar{k})$.  We count the number of orbits of $4$-sets in Section~\ref{Sorbit4}.
This leads to some concluding formulas in Section~\ref{Sconclusions}

\begin{remark}
Most of the material in this paper appeared in the thesis of Ho in 2015 but was not published before.  
We finished this paper after hearing several people 
express interest in moduli spaces of curves with wildly ramified automorphisms.
It is possible that this paper will shed light
on the rational points and cohomology for $B \mu_p/{\mathbb Z}$, 
whose fiber over $p$
is not a Deligne-Mumford stack.
See \cite{kobin} and \cite[Section 4.1]{ESBstack} for a stacky perspective on this.
\end{remark}

\begin{remark}
Other perspectives on the arithmetic statistics of Artin--Schreier curves can be found in recent work.
For example, the papers \cite{BDFL}, \cite{BDFLS}, and \cite{entin}
are about the distribution of the zeros of the $L$-functions.
In \cite{CEZ}, \cite{sankar}, the authors study the question of whether a randomly chosen curve is ordinary, 
or more generally whether the distribution of the $p$-divisible groups of curves matches that 
of the $p$-divisible groups of Dieudonn\'e modules.
\end{remark}

\thanks{
We would like to thank Jeff Achter, Jeremy Booher, Bryden Cais, and Beth Malmskog for many 
inspiring conversations about Artin--Schreier curves.
Pries was supported by 
NSF grants DMS-15-02227 and DMS-19-01819.} 

\section{Background on ramification data and splitting behavior}

We first clarify notation about automorphisms and isomorphisms in Sections~\ref{SautP1}-\ref{SautAS}.
In Sections~\ref{ramificationdata} and \ref{splittingbehavior}, 
we define the ramification data and splitting behavior of Artin--Schreier covers.
We review geometric results about the moduli space ${\mathcal AS}_g$ in Section~\ref{dimension}.
This provides some context for earlier mass formula results of \cite{Nart2, Nart3} in characteristic $p=2$, which we
summarize in Section~\ref{earlierresults}.

Let $k$ be a finite field of cardinality $q=p^n$ and
let $\bar{k}$ be an algebraic closure of $k$.  
Suppose $\pi: C \to {\mathbb P}^1$ is an Artin--Schreier cover over $k$ 
with affine equation $y^p-y=u(x)$ for some $u(x) \in k(x)$.

\subsection{Automorphisms of the projective line} \label{SautP1}

We describe $\gamma \in \textup{Aut}_k(\mathbb{P}^1) \simeq {\rm PGL}_2(k)$ using fractional linear transformations: 
for $a,b,c,d \in k$ with $ad-bc \not = 0$, write
\begin{eqnarray} \label{PGL}
\gamma(x) = \frac{ax+b}{cx+d}.
\end{eqnarray}
The action of $\gamma$ on $u(x) \in k(x)$ is given by $\gamma \cdot u(x) = u(\gamma(x))$. 

We identify the symmetric group $S_3$ with the set of $6$ fractional linear transformations which stabilize $\{0, 1, \infty\}$.
The orbit of $x \in \overline{k}-\{0,1\}$ under the action of $\sigma \in S_3$ is given by
\[\begin{array}{|c|c|c|c|c|c|c|}
\hline
\sigma & {\rm id} & (1,2) & (2,3) & (1,3) & (1,2,3) & (1,3,2) \\
\hline
\sigma(x) & x & 1/x & 1-x & x/(x-1) & 1/(1-x) & (x-1)/x \\
\hline
\end{array}.
\]

The following result is well-known.

\begin{lemma} \label{Lspecialorbit}
The orbits of $x \in \overline{k}-\{0,1\}$ under the $S_3$-action each have size $6$ 
except:
\begin{enumerate}
\item if $p=2$, there is one orbit of size $2$, namely $\{\zeta_3, \zeta_3^2\}$;

\item if $p =3$, there is one orbit of size $1$, namely $\{-1\}$;

\item if $p \geq 5$, there is one orbit of size $3$, namely $\{-1,2,1/2\}$, 
and one orbit of size $2$, namely $\{\zeta_6, \zeta_6^5\}$.
\end{enumerate}
\end{lemma}

For $t \in k-\{0,1\}$, write $B_t=\{0,1,\infty, t\}$. 
Let $\Gamma_t$ be the stabilizer of $B_t$ under the ${\rm PGL}_2(k)$-action.
Note $\Gamma_t \hookrightarrow {\rm Sym}(B_t) \simeq S_4$. 

Define $\gamma_{1,t}, \gamma_{2,t} \in {\rm PGL}_2(k)$ 
	by $\gamma_{1,t}(x) = t/x$ and $\gamma_{2,t}(x) = (x-t)/(x-1)$.
	Then $\gamma_{1,t}$ and $\gamma_{2,t}$ commute, have order 2, 
	and stabilize $B_t$.  
	As permutations of $B_t$, we see that $\gamma_{1,t} = (0,\infty)(1, t)$ and 
	$\gamma_{2,t} = (0,t)(1,\infty)$.  
	Let $K=\langle \gamma_{1,t}, \gamma_{2,t} \rangle \simeq C_2 \times C_2$.  

Let
\[\binom{-3}{q} = \begin{cases} 
1 & \text{ if } q \equiv 1 \bmod 3 \\
-1 & \text{ if } q \equiv -1 \bmod 3 \\
0 & \text{ if } 3 \mid q.
\end{cases}\]

The next lemma can essentially be found in \cite[Table 1]{Leijenhorst}. 

\begin{lemma} \label{C2xC2}
If $t \in k - \{0,1\}$, then $\Gamma_{t} = K$
except in the following cases.
\begin{enumerate}
\item  If $p=2$ and $t \in \{\zeta_3, \zeta_3^2\}$, then $\Gamma_t \simeq A_4$.

The number of orbits with $\Gamma_t=K$ is $(q-3+\binom{-3}{q})/6$.

\item If $p =3$ and $t =-1$, then $\Gamma_t \simeq S_4$.
The number of orbits with $\Gamma_t=K$ is $(q-3)/6$.

\item If $p \geq 5$ and $t \in \{-1,2,1/2\}$, then $\Gamma_t \simeq D_4$.
If $p \geq 5$ and $t \in \{\zeta_6, \zeta_6^5\}$, then $\Gamma_t \simeq A_4$.

The number of orbits with $\Gamma_t=K$ is $(q-6-\binom{-3}{q})/6$.
\end{enumerate}
\end{lemma}

	\begin{proof}
	It is clear that $K \subset \Gamma_t$. 
	If $\Gamma_t \not = K$, then $\Gamma_t$ is isomorphic to $D_4$, $A_4$, or $S_4$.
	This implies that some 
	$\gamma \in \Gamma_t$ fixes an element of $B_t$.
	Without loss of generality, we can suppose that $\gamma$ fixes $t$ and 
	stabilizes $\{0,1,\infty\}$.  This is only possible if $t$ is one of the special values from 
	Lemma~\ref{Lspecialorbit} and that result also shows the following.
	
	If $p=2$ and $t \in \{\zeta_3, \zeta_3^2\}$, then $\Gamma_t$ contains a $3$-cycle and has index $2$ in ${\rm Sym}(B_t)$, and so $\Gamma_t \simeq A_4$.
	The same is true if $p \geq 5$ and $t=\zeta_6, \zeta_6^5$.
	
	If $p=3$ and $t=-1$, then $\Gamma_t$ contains ${\rm Sym}(\{0,1,\infty\})$ in ${\rm Sym}(B_t)$,
	and so $\Gamma_t \simeq S_4$.
	
	If $p \geq 5$ and $t \in \{-1,2,1/2\}$ then $\Gamma_t$ contains a $2$-cycle and has index $3$ in ${\rm Sym}(B_t)$,
	and so $\Gamma_t \simeq D_4$.
	\end{proof}

\subsection{Isomorphisms} \label{Sisomorphisms}

There are different kinds of isomorphisms between Artin--Schreier objects 
and we clarify our definitions of these.
Define ${\rm AS}(k(x)) := \{z^p - z \mid z \in k(x) \}$.  
For $i=1,2$, suppose $\pi_i: C_i \to {\mathbb P}^1$ is given by the equation
$y_i^p-y_i=u_i(x)$ with the automorphism $\iota_i:(x,y_i) \mapsto (x, y_i+1)$. 
 
{\bf Artin--Schreier covers:}
We say that $\pi_1$ and $\pi_2$ are isomorphic over $k$ as Artin--Schreier covers if there 
exists an automorphism $\varphi:C_1 \to C_2$ over $k$ such that $\pi_2 \circ \varphi = \pi_1$.
This is equivalent to the condition $u_2(x) - u_1(x) =z^p-z \in {\rm AS}(k(x))$ for some $z \in k(x)$; 
in this case, the isomorphism is given by the change of variables $y_2=y_1+z$.
In particular, $C$ is geometrically irreducible if and only if $u(x) \notin {\rm AS}(k(x))$. 

{\bf Artin--Schreier curves with automorphism:}
We say that the pairs $(C_1, \iota_1)$ and $(C_2, \iota_2)$ are isomorphic over $k$ 
as Artin--Schreier curves with automorphism if there exists an isomorphism 
$\varphi: C_1 \to C_2$ over $k$ such that 
$\iota_2 \circ \varphi = \varphi \circ \iota_1$.
This is true if and only if $\varphi$ descends to the quotient ${\mathbb P}^1$, meaning that 
there exists $\gamma \in {\rm Aut}_k({\mathbb P}^1)$ such that $\pi_2 \circ \varphi = \gamma \circ \pi_1$.
Viewing $\gamma$ as a change of coordinates on the variable $x$, then
$(C_1, \iota_1)$ and $(C_2, \iota_2)$ are isomorphic over $k$ if and only if
$u_2(x) \equiv \gamma \cdot u_1(x) \bmod {\rm AS}(k(x))$.  

In particular, an automorphism of an Artin--Schreier curve $(C, \iota)$ with automorphism is an isomorphism
$\varphi: C \to C$ such that $\iota \circ \varphi = \varphi \circ \iota$ or, equivalently, $\varphi \in \cent$.

An Artin--Schreier curve with automorphism determines an Artin--Schreier cover $\pi:C \to {\mathbb P}^1$, but 
this is only well-defined up to a change of variables on ${\mathbb P}^1$.

{\bf Artin--Schreier curves:}
We say that $C_1$ and $C_2$ are isomorphic over $k$ as Artin--Schreier curves if 
there exists an isomorphism $\varphi: C_1 \to C_2$ over $k$. 
This is the case if there exists 
$c \in \FF_p^*$ and $\gamma \in {\rm Aut}_k({\mathbb P}^1)$ such that 
$u_2(x) \equiv c \gamma \cdot u_1(x) \bmod {\rm AS}(k(x))$.

For example, if $u_1(x)=x^e$ and $u_2(x) = c x^{e}$, then 
$C_1$ and $C_2$ are isomorphic over $k$ as Artin--Schreier curves, where $\varphi$ 
identifies $y_2=cy_1$; but not as Artin--Schreier curves with automorphism since $\varphi$ identifies
$\iota_2 = \iota_1^c$.

\subsection{Remarks} \label{SautAS}

Suppose $C$ is an Artin--Schreier curve.

When $p=2$, there is a unique choice for $\iota$, namely, the hyperelliptic involution of $C$; 
by \cite[Theorem 11.98]{Hirschfeld}, $\iota$
is in the center of ${\rm Aut}_k(C)$.  In other words, $\cent = {\rm Aut}_k(C)$.  
Also, the data of an Artin--Schreier curve is equivalent to the data of an 
Artin--Schreier curve with automorphism.  So the mass formula $Z_g(q)$ in \eqref{Eweightsum} 
specializes to \eqref{Defa0} when $p=2$.

When $p$ is odd, there are two subtleties with the definition of the mass formula.  
The first is that we weight by $\cent$; this is because 
some automorphisms of $C$ may not descend to the quotient projective line, as seen in the following result.

\begin{lemma} \cite[Theorem 11.93]{Hirschfeld} \label{exceptions}
Suppose $p$ is odd and $g >1$.
Then ${\rm Aut}_k(C)$ is an extension of $\langle \iota \rangle$ by a finite group of $k$-automorphisms 
of $k(x)$ except in the following cases:
		\begin{enumerate}
			\item $y^p - y = a/(x^p-x)$ for $a \in k^*$,
			\item $y^3 - y = b/x(x-1)$ with $b^2 = 2$, or
			\item $y^p - y = 1/x^c$ with $c \mid (p+1)$.
		\end{enumerate}
\end{lemma}

The second subtlety is that we consider the weighted sum over 
isomorphism classes of Artin--Schreier curves with automorphism.
The data of the automorphism is useful for defining the ramification data as in Section~\ref{ramificationdata}.
Note that each Artin--Schreier curve has only finitely many possible choices for $\iota$.
So the geometric facts in Section~\ref{dimension} about the dimension of ${\mathcal AS}_g$ and the number of its components
apply equally well in the context of Artin--Schreier curves with automorphism.

\subsection{Ramification Data} \label{ramificationdata}

If $u(x)$ has a pole at a point $Q \in {\mathbb P}^1(\bar{k})$, let $e_Q$ be the order of the pole and 
$\epsilon_Q = e_Q +1$.
Without loss of generality, we can assume that $p \nmid e_Q$; 
this may require a change of coordinates of the form $y \mapsto y+z$ 
and uses the fact that every element of ${\mathbb F}_q$ is a $p$th power. 
Then the branch locus of $\pi$ is the set of poles of $u(x)$.
We denote by $D$ the degree of the ramification divisor of $\pi$. 

\begin{prop} (Wild Riemann-Hurwitz formula) 
see \cite[Proposition 3.7.8]{Stichtenoth} \label{ramification divisor}
Let $\pi$ be an Artin--Schreier cover with equation $y^p-y=u(x)$ where $u(x) \in k(x) - \textup{AS}(k(x))$.
Suppose ${\rm div}_\infty(u(x)) = \sum e_Q Q$ where each $e_Q$ is a positive prime-to-integer.
Then 
\[2g = -2(p-1) + D, \ {\rm where} \ D =\sum (e_Q + 1)(p-1).\]
\end{prop}

Let $r$ be the number of poles of $u(x)$.
We label the branch points of $\pi$ by $Q_1, \ldots, Q_r$. 
Let $e_i$ be the order of the pole of $u(x)$ at $Q_i$ and $\epsilon_i = e_i + 1$.
The {\it ramification data} of $\pi$ is the tuple $\vec{\epsilon} = (\epsilon_1,\epsilon_2,\cdots, \epsilon_r)$;
without the labeling of the branch points, it can be viewed as a multiset $R=\{\epsilon_1,\epsilon_2,\cdots, \epsilon_r\}$. 

For a fixed genus $g$, there can be several possibilities for the ramification data $R$.
By Proposition~\ref{ramification divisor}, these correspond to partitions of $2+2g/(p-1)$ 
into numbers $\epsilon_i \not \equiv 1 \bmod p$.  

\begin{example}
If $g = 2(p-1)$, 
then the possibilities for $R$ are $\{6\}$ (if $p \not = 5$), 
$\{4,2\}$ (if $p \not = 3$), $\{3,3\}$ (if $p \not = 2$), and $\{2,2,2\}$.
\end{example}

\subsection{Splitting Behavior} \label{splittingbehavior}

Let ${\rm Fr}(\alpha) = \alpha^q$ denote the absolute Frobenius of $\bar{k}$ over $k$.
If $u(x) \in k(x)$, the poles of $u(x)$ are not necessarily in $k$;
it is possible that the set of poles includes an orbit under Frobenius of points defined over 
a finite extension of $k$.
In this case, the orders of the poles of $u(x)$ at these points are the same. 

The {\it splitting behavior} $S$ is the data of the fields of definition of the poles of $u(x)$.
The branch locus is {\it split} if all the poles of $u(x)$ are in $k$.
If not, we need some notation to keep track of the field of definition of the branch points.
For example, when $r=2$, we denote the split case by $(\epsilon, \epsilon)$ 
and the non-split case by $((\epsilon - \epsilon))$.
A point of degree $m$ is the Frobenius orbit of $m$ distinct points defined over $\FF_{q^m}$;
to denote this, we replace the $m$ values of $\epsilon$ in the ramification data by $(\epsilon - \cdots - \epsilon)$.

Let $s$ be the number of orbits under Frobenius of the set of poles of $u(x)$.
Let $m_1, \ldots, m_s$ denote the degrees of representatives of points in the orbits; 
then $\sum_{i=1}^s m_i = r$.

\begin{example}
For $3$ branch points with the same pole order, the splitting behavior can be:\\
(i) $(\epsilon, \epsilon, \epsilon)$, meaning 3 points of degree $1$; \\
(ii) $(\epsilon, (\epsilon-\epsilon))$, meaning $1$ point of degree $1$ and $1$ point of degree $2$; \\
(iii) or $((\epsilon - \epsilon-\epsilon))$ meaning $1$ point of degree $3$. 
\end{example}

\subsection{Dimension of $p$-rank strata of moduli space} \label{dimension}

Let $d \geq 1$.
Let ${\mathcal AS}_{g}$ denote the moduli space of Artin--Schreier curves of genus $g=d(p-1)/2$. 
There are geometric results about ${\mathcal AS}_{g}$;
for example, the number of its irreducible components and their dimensions are known.

The moduli space ${\mathcal AS}_{g}$ can be stratified by the number of branch points 
$r$ of the Artin--Schreier cover or, equivalently, by the $p$-rank.
Here the {\it $p$-rank} of $C$ is the integer $\sigma$ such that $p^\sigma$ 
is the number of $p$-torsion points on ${\rm Jac}(C)$;
equivalently, it is the number of slopes of $0$ in the Newton polygon of $C$.
By the Deuring-Shafarevich formula, $\sigma=(r-1)(p-1)$.

Consider the weighted count $Z_g(q) := a({\mathcal AS}_g)_0\mid_g$ as in \eqref{Eweightsum}.  
If $Z_g(q)$ is a polynomial in $q$, 
then its degree is the dimension of ${\mathcal AS}_g$ and its leading coefficient is the number of irreducible components of ${\mathcal AS}_g$ having that dimension. 

\begin{theorem}\cite[Theorem 1.1]{Pries} \label{Pries}
Let $g = d(p-1)/2$ and $\sigma=(r-1)(p-1)$ for $d, r \geq 1$.
Let $\mathcal{AS}_{g,\sigma}$ denote the $p$-rank $\sigma$ stratum of the moduli space of 
Artin--Schreier curves of genus $g$.  
\begin{enumerate}
	\item The set of irreducible components of $\mathcal{AS}_{g,\sigma}$ is in bijection with the set of partitions 
	$R=\{\epsilon_1, \epsilon_2, \ldots, \epsilon_{r}\}$ of $d+2$ into $r$ positive integers such that each $\epsilon_i \not\equiv 1 \bmod p$.
	\item The irreducible component of $\mathcal{AS}_{g,\sigma}$ for the partition $R=\{ \epsilon_1, \epsilon_2, \ldots, \epsilon_{r}\}$ has dimension \[\delta_R = d-1-\sum_{i=1}^{r}\lfloor \frac{\epsilon_i-1}{p}\rfloor.\]
\end{enumerate}
\end{theorem}

\begin{remark}
In particular, $C$ is ordinary (meaning that $\sigma=g$) if and only if $R=\{2, \ldots, 2\}$.
It follows directly from Theorem~\ref{Pries} that the ordinary locus is not open and dense in 
$\mathcal{AS}_{g}$ if $p \geq 3$ and $d \geq 2$ (with $d \geq 6$ if $p = 3$).
See \cite[Corollary 3.5]{sankar} for another perspective on this.
\end{remark}

\begin{example} \label{Edim6}
If $g = 2(p-1)$, then $d=4$ and here are the dimensions for each partition:
\begin{center}
\begin{tabular} {|c|c|c|c|c|} \hline
$R$ & $\{2,2,2\}$ & $\{3,3\}, \ p \not = 2$ & $\{4,2\}, \ p \not = 3$ & $\{6\}, \ p \not = 5$ \\ \hline
$\delta_R$ & $3$ & $3$ & $3 \ {\rm if } \ p \geq 5$ & $3 \ {\rm if } \ p \geq 7$ \\  
& & & $2 \ {\rm if } \ p = 2$ &  $2 \ {\rm if } \ p =3$ \\
&&&&  $1 \ {\rm if } \ p =2$ \\ \hline
\end{tabular}
\end{center}
\end{example}

\subsection{Earlier Results when $p = 2$} \label{earlierresults}

We describe some earlier results when $q$ is a power of $p=2$ and $g=2,3$.
In these results, $Z_g(q)$ was computed by separating into 
cases based on the ramification data $R$ and splitting behavior $S$.  (We combine the various splitting behaviors for brevity).   
These results were generalized
for any genus $g$ when $p=2$ using a different method \cite[Sections 8-10]{bergstrom}.
However, the case-based method is still valuable since it
illustrates the connection with the geometry of the $p$-rank strata as discussed in Section~\ref{dimension}.
 
\begin{theorem}\cite[Theorem 18]{Nart2}
If $k$ is a finite field of cardinality $q=2^n$ and $g=2$, then
\[Z_2(q) = \sum_{[C] \in {\mathcal AS}_2(k)/\simeq_k}|\textup{Aut}_k(C)|^{-1} = q^3.\]
This result was obtained by combining the following data.
		\begin{center}
		\begin{tabular}{|c|c|c|}
		\hline
		$R$ & $\delta_R$ & Mass formula \\
		\hline
		$\{2,2,2\}$ & 3 & $q^3-q^2$ \\ \hline 
		$\{2,4\}$ & 2 & $q^2-q$\\ \hline
		$\{6\}$ & 1 & $q$\\
		\hline
		\end{tabular}
		\end{center}
\end{theorem}

\begin{theorem} \label{Nq5} \cite[Theorem 8]{Nart3}
If $k$ is a finite field of cardinality $q=2^n$ and $g=3$, then
\[Z_3(q) = \sum_{[C] \in {\mathcal AS}_3(k)/\simeq_k} |\textup{Aut}_k(C)|^{-1} = q^5.\]
This result was obtained by combining the following data.
\begin{center}
		\begin{tabular}{|c|c|c|}
		\hline
		$R$ & $\delta_R$ & Mass formula \\
		\hline
		$\{2,2,2,2\}$ & $5$ & $q^5-q^4$ \\ \hline
		$\{2,2,4\}$ & 4 & $q^4 - 2q^3 + q^2$ \\ \hline
		$\{4,4\}$ & 3 & $q^3-q^2$ \\ \hline
		$\{2,6\}$ & 3 & $q^3-q^2$\\ \hline
		$\{8\}$ & 2 & $q^2$\\
		\hline
		\end{tabular}
		\end{center}
\end{theorem}

\section{Counting Artin-Schreier equations}

The main result of this section is Proposition~\ref{PcountN}, 
in which we determine the number of equations for an Artin--Schreier cover with 
fixed ramification data, splitting behavior and branch divisor.
This relies on Lemma~\ref{partialfraction}, which provides a useful way to describe the equation 
for an Artin--Schreier cover when the branch points are not rational. 

\subsection{Orbit of rational functions under Frobenius}

Recall that ${\rm Fr}(\alpha) = \alpha^q$ is the absolute Frobenius of $\bar{k}$ over $k$.
If the branch points of $\pi$ are not defined over $k$, the condition that $u(x) \in k(x)$ places constraints on the partial fraction decomposition of $u(x)$ over $\overline{k}$.  In particular, the terms in the partial fraction 
decomposition must respect the action of ${\rm Fr}$.  

For example, suppose $q \equiv 3 \bmod 4$ and $u(x)=(x^2+1)^{-1}$.  
Let $i$ denote $\sqrt{-1} \in {\mathbb F}_{q^2}$. 
Then 
\[u(x) = \frac{i/2}{x+i} + \frac{-i/2}{x-i} = \frac{i/2}{x+i} + {\rm Fr}(\frac{i/2}{x+i}).\]
The next lemma generalizes this example.

\begin{lemma}\label{partialfraction}
Let $u(x) \in {\mathbb F}_q(x)$ be such that ${\rm div}_{\infty}(u(x)) = e P$, where $P$ is an 
${\mathbb F}_q$-point of degree $m$.  
Then there exist a constant $c_0 \in {\mathbb F}_q$, a rational function $v(x) \in {\mathbb F}_{q^m}(x)$, 
and an ${\mathbb F}_{q^m}$-point $P_0$ of degree one such that 
${\rm div}_{\infty}(v(x)) = e P_0$ 
and \[u(x) = c_0 + \sum_{j =0}^{m-1} {\rm Fr}^j(v(x)).\] 
\end{lemma}

\begin{proof}
The ${\mathbb F}_q$-point $P$ consists of an orbit $\{\theta_j\}_{0 \leq j \leq m-1}$ of ${\mathbb F}_{q^m}$-points under Frobenius, where ${\rm Fr}(\theta_j) = \theta_{j+1 \bmod m}$. 
Let $P_0 = \theta_0$. 
The partial fraction decomposition of $u(x)$ over 
${\mathbb F}_{q^m}$ has the form 
\[u(x) = c_0 + \sum_{j = 0}^{m-1} \frac{v_j(x)}{(x-\theta_j)^e},\]
where $v_j(x) \in {\mathbb F}_{q^m}[x-\theta_j]$ is a polynomial in $x - \theta_j$ of degree at most $e-1$ with non-zero constant term.
Thus 
\[u(x)={\rm Fr}(u(x)) = {\rm Fr}(c_0) + \sum_{j=0}^{m-1} \frac{{\rm Fr}(v_j(x))}{(x-\theta_{j+1})^{\epsilon}}.\] 
Matching up the terms shows that ${\rm Fr}(c_0) = c_0$ and 
${\rm Fr}(v_j(x))=v_{j+1 \bmod m}(x)$ for $0 \leq j \leq m-1$.  The result follows by letting $v(x) =  v_0(x)/(x-\theta_{0})^e$. 
\end{proof}

\subsection{The constant term}

\begin{lemma} \label{constant coefficient}
Fix $u(x) \in k(x) - {\rm AS}(k(x))$.  
For $c \in k$, let $\pi_c: C_{c} \to {\mathbb P}^1$ be given by the equation $y^p-y = u(x) + c$.  
Then there are exactly $p$ isomorphism classes of Artin--Schreier covers of the form $\pi_c$.
\end{lemma}

\begin{proof}
Two such covers $\pi_{c_1}$ and $\pi_{c_2}$ are isomorphic if and only if $c_1 - c_2 = z^p - z \in {\rm AS}(k(x))$ for some $z \in k$.
So the number of isomorphism classes equals the number of cosets of ${\rm AS}(k)$ in $k$.  
Since there are $q/p$ elements in ${\rm AS}(k)$, the number of cosets of $\textup{AS}(k)$ in $k$ is $p$.
\end{proof}

\subsection{Counting rational functions} \label{rationalequations}

Given a pole $\theta_0$ of $u(x)$, 
we find the number of possibilities for its contribution to the partial fraction decomposition of $u(x)$.  
This number does not depend on the location of the pole but only on the order $e$ of the pole at $\theta_0$ 
and the degree $m$ of the point $\theta_0$.  

Given $\theta_0 \in {\mathbb P}^1({\mathbb F}_{q^m})$, write $\bar{x}=x$ if $\theta_0 = \infty$ and 
write $\bar{x} = (x-\theta_0)^{-1}$ otherwise.
Recall the notation in Lemma~\ref{partialfraction}.
We write $v(x) = \sum_{i=1}^e a_i \bar{x}^i$ for some constants $a_i \in \FF_{q^m}$.
Without loss of generality, we can adjust 
$u(x)$ over ${\mathbb F}_q$ to remove any monomial $a_i \bar{x}^i$ in $v(x)$ with $p \mid i$.
We count the remaining number of possibilities for $v(x)$.

\begin{defn}
Given $m$ and $e$, let $V_{m,e}$ be the subset of
$\{v \in \bar{x}{\mathbb F}_{q^m}[\bar{x}] \mid \deg(v) = e \}$ 
consisting of polynomials $v$ that contain no monomials $a_i \bar{x}^i$ with $p \mid i$.
\end{defn}  

\begin{lemma} \label{counting} 
Then $\#V_{m,e}=(q^m-1)q^{m\left(e-1-\lfloor \frac{e}{p} \rfloor\right)}$.
\end{lemma}

\begin{proof}
Write $v = \sum_{i=1}^{e} a_i \bar{x}^i$ where $a_i \in {\mathbb F}_{q^m}$.  Since $\deg(v)=e$, there are $q^m-1$ choices for the leading coefficient $a_{e}$ and $q^m$ choices for the coefficient $a_i$ if $1 \leq i < e$.  
Also, the number of monomials is $\#\{1 \leq i \leq e-1 \mid p \nmid i\} = e-1-\lfloor \frac{e}{p} \rfloor$.
\end{proof}

\subsection{Counting Artin--Schreier covers with fixed branch divisor}

Fix a ramification data $R$ of length $r$ as in Section~\ref{ramificationdata}
and a splitting behavior $S$ as in Section~\ref{splittingbehavior}.
We count the number of Artin--Schreier covers of type $R,S$ and
with fixed branch divisor.

Write 
$R := \{\epsilon_1, \epsilon_2, \cdots, \epsilon_r\}$.
Recall that $e_i = \epsilon_i -1$ for $1 \leq i \leq r$. 
Define 
\begin{equation} \label{EdefE}
E := E(R) = \sum_{i=1}^r \left(e_i - 1 - \left\lfloor \frac{e_i}{p} \right\rfloor\right).
\end{equation}

Fix a weighted branch divisor $W_\circ = \sum e_{Q} \cdot Q$ of type $R,S$.
Let $B_\circ$ be its support.  
Recall that $s$ is the number of orbits of the points in $B_\circ$ under Frobenius.
Recall that $m_1, \ldots, m_s$ are the degrees of these points and $\sum_{i=1}^s m_i = r$.

Define $N_{W_\circ}$ to be the set of $u(x) \in k(x)$ such that ${\rm div}_\infty(u(x))  = W_\circ$
modulo ${\rm AS}(k(x))$.

\begin{prop} \label{PcountN}
With notation as above, the number of Artin--Schreier covers $y^p-y=u(x)$ with ${\rm div}_\infty(u(x)) = W_\circ$ is
\[|N_{W_\circ}| = p \cdot q^{E(R)}\prod_{i=1}^s (q^{m_i} - 1).\]
\end{prop}

\begin{proof}
It suffices to count the number of $u(x) \in \FF_q(x)$ of type $R,S$ such that ${\rm div}_\infty(u(x))  = W_\circ$ up to Artin--Schreier equivalence.
The partial fraction decomposition of $u(x)$ has the form 
$u(x) = \sum_{i=1}^s u_i(x)$ where $u_i(x) \in \FF_q(x)$ 
and ${\rm div}_\infty(u_i(x)) = e_i P_i$ for some $\FF_q$-point $P_i$ of degree $m_i$.
Using Lemma~\ref{partialfraction}, we write 
$u_i(x) = c_{0,i} + \sum_{j =0}^{m-1} {\rm Fr}^j(v_i)$ with $v_i \in V_{m_i,e_i}$.  
There are $q$ choices for the combined constant term $\sum_{i=1}^s c_{0,i}$.
By Lemma~\ref{constant coefficient}, the number of choices for the constant term up to Artin--Schreier equivalence is $p$. 
So $|N_{W_\circ}| = p \prod_{i=1}^s |V_{m_i, e_i}|$.
This simplifies to the given formula by Lemma~\ref{counting}. 
\end{proof}

\section{The mass formula as an average}

Fix $g=d(p-1)/2$.  
Fix a ramification data $R$ of length $r$ as in Section~\ref{ramificationdata}
and a splitting behavior $S$ as in Section~\ref{splittingbehavior}.
Given an Artin--Schreier curve $(C, \iota)$ with automorphism, 
the ramification data and splitting behavior of the resulting Artin--Schreier cover are well-defined, 
because they do not depend on the choice of the parameter on the quotient ${\mathbb P}^1$.

Define a mass formula
\[Z_{R,S}(q) : = \sum_{[C] \in {\mathcal AS}_g(k)/\simeq_k, \ {\rm type } \ R, S} |\cent|^{-1},\]
where the sum is over all isomorphism classes $[C]=(C,\iota)$ of Artin--Schreier curves with automorphism over $k$ 
such that the cover has type $R, S$.
The main result in this section is Theorem~\ref{weightedsum}, in which we 
re-interpret this mass formula 
as an average over conjugacy classes in the symmetric group $S_r$.
This result was inspired by \cite[Corollary 5.3]{VanDerGeer}.
 
\subsection{Stabilizers} \label{group actions}

Suppose $W=\sum e_{Q} Q$ is a weighted branch divisor of type $R,S$.

\begin{defn}
Let $\Gamma_W= \{\gamma \in {\rm Aut}_k({\mathbb P}^1) \mid \gamma(W) = W\}$.
\end{defn}

For example, consider the case that $r \leq 3$.  Then we can fix a convenient choice 
of branch locus $B$ because
the action of ${\rm PGL}_2(k)$ is triply transitive on the points of degree $1$,
and is transitive on the points of degree $2$ (resp.\ 3). 
We fix a point $\omega_2$ of degree $2$ and a point $\omega_3$ of degree $3$.
We compute the order of the stabilizer $\Gamma_W$ in these cases.

Let $\epsilon_1, \epsilon_2, \epsilon_3$ be distinct positive integers with $p \nmid e_i = \epsilon_i-1$.
Write $\epsilon=\epsilon_1$.

\begin{lemma} \label{LgammaW123}
The following table provides the order of the stabilizer $\Gamma_W$ in each case.
\begin{center}
\begin{tabular}{lllll}\hline
$r$ & $R$ and $S$ & $B$ & $|\Gamma_W|$\\ \hline
1 & $(\epsilon)$  & $\{\infty\}$ & $q(q-1)$\\
2 & $(\epsilon,\epsilon_2)$  & $\{0,\infty\}$ & $q-1$\\
2 & $(\epsilon,\epsilon)$  & $\{0,\infty\}$ & $2(q-1)$\\
2 & $((\epsilon - \epsilon))$ & $\{\omega_2\}$ & $2(q+1)$\\
3 & $(\epsilon,\epsilon_2, \epsilon_3)$ & $\{0,1, \infty\}$ & 1\\
3 & $(\epsilon,\epsilon_2, \epsilon_2)$ & $\{0,1, \infty\}$ & 2\\
3 & $(\epsilon,\epsilon, \epsilon)$ & $\{0,1, \infty\}$ & 6\\
3 & $(\epsilon,(\epsilon_2 - \epsilon_2))$ or $(\epsilon, (\epsilon -\epsilon))$ & $\{\infty, \omega_2 \}$ & 2\\
3 & $((\epsilon - \epsilon - \epsilon))$ & $\{\omega_3\}$ & 3\\ \hline
\end{tabular}
\end{center}
\end{lemma}

\begin{proof}
For $r = 1$, then $W=e \{\infty\}$ and so 
$\Gamma_W=\Gamma_\infty:=\{x \mapsto a x + b \mid a \in k^*, b \in k\}$.

For $r = 2$, $(\epsilon,\epsilon_2)$, then 
$W =e \{0\} + e_2 \{\infty\}$ and so $\Gamma_W=\{x \mapsto a x \mid a \in k^* \}$.  

For $r=2$, $(\epsilon,\epsilon)$, then 
$W =e \{0\} +e\{\infty\}$ and so \[\Gamma_W = \{x \mapsto a x \mid a \in k^* \} \cup 
\{x \mapsto b/x \mid b \in k^* \}.\]

For $r=2$, $((\epsilon-\epsilon))$, then $W = e \{\omega_2\}$.
Let $T_2$ be the set of points of degree $2$ over $k$.  Then $|T_2|=(q^2-q)/2$. 
Since the action of ${\rm PGL}_2(k)$ on $T_2$ is transitive, 
$|\Gamma_W| = |{\rm PGL}_2(k)|/\#T_2 = 2(q+1)$, by the orbit-stabilizer theorem.

For $r = 3$, in the split cases, then $B=\{0,1,\infty\}$.
For $(\epsilon,\epsilon_2,\epsilon_3)$, then $\Gamma_W = \rm{Id}$.  
For $(\epsilon,\epsilon_2,\epsilon_2)$, let $W = e\{1\} + e_2\{0\} + e_2\{\infty\}$.
Then $\Gamma_W$ also includes $\{x \mapsto 1/x \}$. 
For $(\epsilon,\epsilon,\epsilon)$, there are 6 fractional linear transformations fixing 
$W=e\{0\} + e\{1\} + e\{\infty\}$.

For the two quadratic cases with $r=3$,  
consider $B=\{\infty, \theta_2\}$, for an arbitrary point $\theta_2$ of degree $2$.
There are exactly two automorphisms $\gamma \in {\rm PGL}_2({\mathbb F}_{q^2})$
that fix $\infty$ and stabilize $\theta_2$; these are defined over ${\mathbb F}_q$ and 
so are contained in $\Gamma_\infty$. 
Furthermore, because $|\Gamma_\infty|=2|T_2|$, the action of 
$\Gamma_\infty$ on the set of branch divisors $\{\infty, \theta_2\}$ is transitive.
 
For the cubic case with $r=3$, then $W=e \{\omega_3\}$.
Let $T_3$ be the set of points of degree $3$ over $k$.  Then $|T_3| = (q^3-q)/3$.
Since the action of ${\rm PGL}_2(k)$ on $T_3$ is transitive,
$|\Gamma_{\omega_3}| = |{\rm PGL}_2(k)|/|T_3|=3$, by the orbit-stabilizer theorem.
\end{proof}

\subsection{The action on rational functions}

If $\gamma \in {\rm PGL}_2(k)$, then $\gamma$ acts on rational functions in $k(x)$.
Specifically, $\gamma \cdot f(x) = f(\gamma(x))$.

\begin{defn} If $u(x) \in k(x)$, let
$\Gamma_{u(x)} = \{ \gamma \in  {\rm Aut}_k({\mathbb P}^1) \mid u(\gamma(x)) = u(x) \}$.
\end{defn}

If ${\rm div}_\infty(u(x)) = W$, then $\Gamma_{u(x)} \subset \Gamma_W$.

\begin{lemma} \label{exact sequence}
Suppose $\pi: C \to \mathbb{P}^1$ is an Artin--Schreier cover with equation $y^p-y=u(x)$.
If $\varphi \in \cent$, then there exists an automorphism $\gamma: \mathbb{P}^1 \rightarrow \mathbb{P}^1$ 
such that $\pi \circ \varphi = \gamma \circ \pi$.
The map $\psi: \cent \to \textup{Aut}(\mathbb{P}^1)$ is a homomorphism with image $\Gamma_{u(x)}$.
There is a short exact sequence of groups:
\begin{equation} \label{SEScent}
1 \rightarrow \langle \iota \rangle  \overset{I}\rightarrow \cent \overset{\psi}\rightarrow \Gamma_{u(x)} \to 1.
\end{equation}
\end{lemma}

\begin{proof}
Suppose $\varphi \in \cent$.
For $x \in \mathbb{P}^1$, choose $y \in \pi^{-1}(x)$, and define $\gamma(x)=\pi(\varphi(y))$.
We check that $\gamma$ is well-defined: if $y' \in \pi^{-1}(x)$, then $y'=\iota^e(y)$ for some $0 \leq e \leq p-1$;
then $\varphi(y') = \varphi(\iota^e(y)) = \iota^e (\varphi(y))$ so $\pi(\varphi(y')) = \pi(\varphi(y))$.
One can check that $\gamma$ is an automorphism and $\psi$ is a homomorphism.

By Section~\ref{Sisomorphisms}, the condition $\varphi \in \cent$ implies that $u(\gamma(x)) = u(x)$,
so $\gamma \in \Gamma_{u(x)}$.
Conversely, if $\gamma \in \Gamma_{u(x)}$, then the map $(x,y) \mapsto (\gamma(x), y)$ defines an 
automorphism $\varphi \in \textup{Aut}(\mathbb{P}^1)$ and $\varphi \in \cent$.
Thus $\textup{Im}(\psi) = \Gamma_{u(x)}$. 
The sequence in \eqref{SEScent} is exact because $\varphi \in \ker(\psi)$ if and only if $\pi(\varphi(y))=\pi(y)$ if and only if $\varphi \in \langle \iota \rangle$.
\end{proof}

By Lemma \ref{exact sequence}, \begin{equation} \label{Epfactor}
|\cent| = p |\Gamma_{u(x)}|.
\end{equation} 

\subsection{Orbits of weighted branch divisors}

Fix a ramification data $R$ of length $r$ and a splitting behavior $S$.
An Artin--Schreier cover $\pi:C \to {\mathbb P}^1$ of type $R,S$
determines a weighted branch divisor $W = \sum e_{Q} Q$ of type $R,S$, namely
$W={\rm div}_\infty(u(x))$ where $y^p-y=u(x)$ is the equation for $\pi$.
 
Fix one weighted branch divisor $W_\circ$ of type $R,S$.
Let $\Gamma_{W_\circ}$ be the stabilizer of $W_\circ$ in ${\rm PGL}_2(k)$.
Recall that $N_{W_\circ}$ is the set of $u(x) \in k(x)$ such that ${\rm div}_\infty(u(x))  = W_\circ$
modulo ${\rm AS}(k(x))$.
Note that $\Gamma_{W_\circ}$ acts on $N_{W_\circ}$.  
By the orbit-stabilizer theorem, 
\begin{equation} \label{Estab}
|\textup{Orb}_{\Gamma_{W_\circ}}(u(x))| |\Gamma_{u(x)}| = |\Gamma_{W_\circ}|.
\end{equation}

If two Artin--Schreier curves with automorphism are isomorphic over $k$, 
then their weighted branch divisors $W_1$ and $W_2$ are in the same orbit under ${\rm PGL}_2(k)$.
Thus the stabilizers $\Gamma_{W_1}$ and $\Gamma_{W_2}$ are conjugate in the symmetric group $S_r$.

\begin{defn}
Consider the set of weighted branch divisors $W$ over $k$ of type $R, S$.
Let $\theta$ be the set of orbits of $W$ under ${\rm PGL}_2(k)$.
Given $H \subset S_r$, let $\theta_H$ be the set of orbits for $W$ under ${\rm PGL}_2(k)$
having the property that $\Gamma_{W}$ is conjugate to $H$.
Let $T=\sum_{H} \frac{|\theta_H|}{|H|}$, where the sum 
ranges over a set of representatives of conjugacy classes of subgroups $H \subset S_r$. 
\end{defn}

\begin{example} \label{Etheta3}
If $r \leq 3$, for given $R$ and $S$, recall the choice of weighted branch divisor $W$ and $\Gamma_W$ from (the proof of) Lemma~\ref{LgammaW123}.
Note that $\theta_H$ is empty unless $H$ is conjugate to $\Gamma_W$, in which case
$|\theta_H|=1$.  Thus $T = |\Gamma_{W}|^{-1}$ if $r \leq 3$.
\end{example}

\subsection{Main Theorem} \label{main theorems}

\begin{theorem} \label{weightedsum}
As $[C]=(C, \iota)$ ranges over the $k$-isomorphism classes of Artin--Schreier curves with automorphism 
having type $R,S$, 
then \[\sum_{[C]/\simeq k, \ {\rm type} \ R,S} |\cent|^{-1} = p^{-1} |N_{W_\circ}| \cdot T.\]
\end{theorem}

\begin{proof}
We specify the information needed to define an Artin--Schreier curve with automorphism of type $R, S$.
First, for each orbit in $\theta$, we choose a representative weighted branch divisor $W$ of type $R,S$ in that orbit.
Then we choose
$u(x) \in k(x)$ such that ${\rm div}_\infty (u(x)) = W$ up to Artin--Schreier equivalence.
Then we divide by $|\textup{Orb}_{\Gamma_{W}}(u(x))|$ to avoid over-counting.
So 
\begin{eqnarray*}
\sum_{[C]/\simeq k, \ {\rm type} \ R,S} \frac{1}{|\cent|} & = &  \sum_{W \in \theta} \sum_{u(x)} \frac{1}{|\textup{Orb}_{\Gamma_{W}}(u(x))|\cdot |\cent|}\\
& = & \sum_{W \in \theta} \sum_{u(x)} \frac{1}{|\textup{Orb}_{\Gamma_{W}}(u(x))| \cdot p|\Gamma_{u(x)}|} \\
& = & \frac{1}{p} \sum_{W \in \theta}  \frac{1}{|\Gamma_{W}|} \sum_{u(x)} 1, 
\end{eqnarray*}
by Lemma~\ref{exact sequence} and the orbit-stabilizer theorem.

The number of $u(x)$ such that ${\rm div}_\infty(u(x)) = W$ is independent of $W$ when $R$ and $S$ are fixed.
Up to Artin--Schreier equivalence, this cardinality equals $|N_{W_\circ}|$.
Using a set of representatives of conjugacy classes of subgroups $H \subset S_r$, we re-organize the sum
by combining the contributions from all $W$ such that $\Gamma_W$ is conjugate to $H$.
Since $|\Gamma_W|=|H|$,
this yields 
\[\sum_{[C]/\simeq k, \ {\rm type} \ R,S} \frac{1}{|\cent|} =
\frac{1}{p} \sum_{H} \frac{|\theta_H|}{|H|} |N_{W_\circ}|= p^{-1}  |N_{W_\circ}| \cdot T.\]
\end{proof}

\section{New results for arbitrary prime $p$ and $r \leq 3$ branch points}

In this section, we suppose that the number of branch points of $\pi$ satisfies $r \leq 3$.

\subsection{Arbitrary genus with $r \leq 3$ branch points}

Recall the definition of $E(R)$ from \eqref{EdefE}.

\begin{prop} \label{Results123}
Given $R$ with length $1 \leq r \leq 3$, 
this table gives the mass formula 
\[Z_{R}(q)=\sum_{[C]/\simeq k, \ {\rm type} \ R} |\cent|^{-1}\] 
for $k$-isomorphism classes of Artin--Schreier curves with automorphism with ramification data $R$.
\begin{center} 
\begin{tabular}{ll}
\hline
$R$ & $Z_{R}(q)$\\
\hline
$\{\epsilon\}$ & $q^{E(R)-1}$\\
$\{\epsilon_1, \epsilon_2\}$ & $(q-1)q^{E(R)}$ \\
$\{\epsilon, \epsilon\}$ & $(q-1)q^{E(R)}$\\
$\{\epsilon_1, \epsilon_2, \epsilon_3\}$ & $(q-1)^3q^{E(R)}$ \\
$\{\epsilon_1, \epsilon_2, \epsilon_2\}$ & $(q-1)^2 q^{E(R)+1}$ \\
$\{\epsilon, \epsilon, \epsilon\}$ & $(q-1)q^{E(R)+2}$ \\
\hline
\end{tabular}
\end{center}
\end{prop}

\begin{proof}
For fixed $R$ and $S$, define the branch locus $B$ as in Lemma~\ref{LgammaW123}
and consider a weighted branch divisor $W$ supported on $B$ as in the proof of that result.
By Theorem \ref{weightedsum} and Example~\ref{Etheta3}, $Z_{R,S}(q) = p^{-1} |N_{W}| |\Gamma_W|^{-1}$.
The formula for $|N_{W}|$ is in Proposition~\ref{PcountN}.
This yields the table below.
\begin{center}
\begin{tabular}{lll} \hline 
$R,S$ & $|N_W|$ & $|\Gamma_W|$\\ \hline
$(\epsilon)$ & $p(q-1)q^{E(R)}$ & $q(q-1)$\\ 
$(\epsilon_1,\epsilon_2)$ &$p(q-1)^2q^{E(R)}$ &  $q-1$\\
$(\epsilon, \epsilon)$ & $p(q-1)^2q^{E(R)}$ &  $2(q-1)$\\ 
$((\epsilon - \epsilon))$ & $p(q^2-1)q^{E(R)}$ & $2(q+1)$\\ 
$(\epsilon_1, \epsilon_2, \epsilon_3)$ & $p(q-1)^3q^{E(R)}$ & $1$\\
$(\epsilon_1, \epsilon_2, \epsilon_2)$ & $p(q-1)^3q^{E(R)}$ &  $2$\\
$(\epsilon, \epsilon, \epsilon)$ & $p(q-1)^3q^{E(R)}$ &  $6$\\
$(\epsilon_1, (\epsilon_2 -\epsilon_2))$ 
& $p(q-1)(q^2-1)q^{E(R)}$ &  $2$\\
$(\epsilon,(\epsilon-\epsilon))$ & $p(q-1)(q^2-1)q^{E(R)}$ & $2$\\ 
$((\epsilon - \epsilon - \epsilon))$ & $p(q^3-1)q^{E(R)}$ & $3$\\ \hline\\
\end{tabular}\\
\end{center}

If some values in $R$ repeat, we combine contributions from different splitting behaviors:

If $R=\{\epsilon, \epsilon\}$, the split and non-split cases 
sum to $(q-1)q^{E(R)}$.

If $R=\{\epsilon_1, \epsilon_2, \epsilon_2\}$, the split and split-quadratic case
sum to $(q-1)^2 q^{E(R)+1}$.

If $R=\{\epsilon, \epsilon, \epsilon\}$, the split, split/quadratic, and cubic cases
sum to $(q-1)q^{E(R)+2}$.

\end{proof}

\subsection{Results for moduli of Artin--Schreier curves with $r \leq 3$}

We compute the mass formula when the Artin--Schreier curve has low genus $g$ by allowing the 
ramification data $R$ and splitting behavior $S$ to vary.  
The leading terms of the formulas are compatible with the 
information about irreducible components of ${\mathcal AS}_g$ from Section~\ref{dimension}. 
The lower degree terms provide new and more subtle information about the cohomology of 
${\mathcal AS}_g$.

\begin{theorem} \label{results for p=3}
Let $p=3$.  As $[C]=(C, \iota)$ ranges over the $k$-isomorphism classes of Artin--Schreier curves with automorphism 
such that $C$ has genus $g$, then
	$$Z_g(q) := \sum_{[C] \in {\mathcal AS}_{g}(k) /\simeq k} |\cent|^{-1} =
	\begin{cases}
	1 & \text{if } g = 1\\
	q-1 & \text{if } g = 2\\
	q^2 & \text{if } g = 3\\
	2q^3-q^2 & \text{if } g = 4\\
	q^4 -q^3+q^2 & \text{if } g = 5.\\
	\end{cases}$$
	
This is immediate by combining the following data. 	
\begin{center} \begin{tabular}{|c|c|c|c|}
\hline
$g=d$ & $R$ & $\delta_R$ & Mass formula \\ \hline
1 & \{3\} & 0 & 1\\ \hline
2 & \{2,2\} & 1 & $q-1$\\ \hline
3 & \{2,3\} & 2 & $q^2-q$\\
& \{5\} & 1 & $q$\\ \hline
4 & \{2,2,2\} & 3 & $q^3-q^2$\\
& \{3,3\} & 3 & $q^3-q^2$\\
& \{6\} & 2 & $q^2$\\ \hline
5 & \{2,2,3\} & 4 & $q^4-2q^3+q^2$\\
& \{2,5\} & 3 & $q^3-q^2$\\
& \{7\} & 2 & $q^2$\\ \hline
\end{tabular} \end{center}
\end{theorem}
\begin{proof}
The proof follows from Proposition~\ref{Results123}.
Some lines include several choices of $S$.
\end{proof}

\begin{theorem}  \label{resultsforpnot3}
Let $p \geq 5$.  As $[C]=(C, \iota)$ ranges over the $k$-isomorphism classes of Artin--Schreier curves with automorphism
such that $C$ has genus $g$, then
	$$Z_g(q) := \sum_{[C] \in {\mathcal AS}_{g}(k)/\simeq k} |\cent|^{-1} =
	\begin{cases}
	1 & \text{if } g = 1(p-1)/2\\
	2q-1 & \text{if } g = 2(p-1)/2\\
	2q^2-q & \text{if } g = 3(p-1)/2\\
	3q^3-3q^2 & \text{if } g = 4(p-1)/2, \ p = 5\\
	4q^3-3q^2 & \text{if } g = 4(p-1)/2, \ p \geq 7\\
	3q^4-3q^3+q^2 & \text{if } g = 5(p-1)/2, \ p = 5\\
	4q^4-4q^3+q^2 & \text{if } g = 5(p-1)/2, \ p \geq 7.\\
	\end{cases}$$
	This is immediate by combining the following data. 	
\begin{center} \begin{tabular}{|c|c|c|c|}
\hline
$g$ & $R$ & $\delta_R$ & Mass formula\\ \hline
$1(p-1)/2$ & \{3\} & 0 & 1\\ \hline
$2(p-1)/2$ & \{2,2\} & 1 & $q-1$\\
 & \{4\} & 1 & $q$\\  \hline
$3(p-1)/2$ & \{2,3\} & 2 & $q^2-q$\\
 & \{5\} & 2 & $q^2$\\     \hline
$4(p-1)/2$ & \{2,2,2\} & 3 & $q^3-q^2$\\  
 & \{2,4\} & 3 & $q^3-q^2$\\
 & \{3,3\} & 3 & $q^3-q^2$\\     
 & \{6\} & 3 & $q^3$ if $p \geq 7$\\   \hline
 $5(p-1)/2$ & \{2,2,3\} & 4 & $q^4-2q^3+q^2$\\  
 & \{2,5\} & 4 & $q^4-q^3$\\  
 & \{3,4\} & 4 & $q^4-q^3$\\
 & \{7\} & 3 & $q^3$ if $p = 5$\\
 & & 4 & $q^4$ if $p \geq 7$\\ \hline
\end{tabular}\end{center}
\end{theorem}
\begin{proof}
The proof follows from Proposition~\ref{Results123}.
For example, for the $R=\{2,2,3\}$ line, we combine the contributions
from the split and the quadratic case to obtain
\[(1/2) ((q-1)^3q +(q-1)(q^2-1)q)=q^2(q-1)^2.\]
\end{proof}

\begin{remark}
In Theorem~\ref{resultsforpnot3}, when $g=4(p-1)$, the case $p=5$ is different because
${\mathcal AS}_{g}$ has one additional component of dimension $3$ when $p \geq 7$, 
namely that for $R=\{6\}$.
\end{remark}

\begin{remark}
The ordinary locus (curves with $p$-rank $g$) corresponds to the ramification data $R=\{2,2, \ldots, 2\}$.
When the leading coefficient of $Z_g(q)$ is bigger than $1$, this reflects the fact that ${\mathcal AS}_g$ has 
more than $1$ irreducible component. 
This means that the ordinary locus is not open and dense in ${\mathcal AS}_g$.

For example, when $p=3$ and $g=4$, then the leading coefficient in Theorem~\ref{results for p=3} is $2$;
the $3$-rank is $4$ for the Artin-Schreier curves with $R=\{2,2,2\}$ but is only $2$ for $R=\{3,3\}$. 
\end{remark}

\section{Orbits of $4$-sets} \label{Sorbit4}

Recall that $k$ is a finite field of cardinality $q$.
A {\it 4-set} is a set of $4$ distinct points of ${\mathbb P}^1(\overline{k})$ which is stabilized by the Frobenius action
over $k$.  
In this section, we study the action of ${\rm PGL}_2(k)$ on 4-sets.
The main result is Proposition~\ref{Pnorb}, which gives a formula for
the number of orbits of $4$-sets with fixed splitting behavior.
We expect this material is well-understood by experts in finite geometry, but we 
did not find a reference that applied directly.
In particular, we extend results found in \cite{NSet}, \cite{Goppa}, and \cite{Nart3}.  

In Section~\ref{Sweightsum4}, we use Proposition~\ref{Pnorb} to 
determine the mass formula when the Artin-Schreier cover
has $r=4$ 
branch points, after 
determining additional information about the sizes of the stabilizers 
of representatives of the orbits.

\subsection{The total number of orbits} \label{orbits}

Let $N_{orbit}$ be the total number
of orbits of 4-sets of $\mathbb{P}^1(\overline{k})$
under ${\rm PGL}_2(k)$.
Let
\[\binom{-3}{q} = \begin{cases} 
1 & \text{ if } q \equiv 1 \bmod 3 \\
-1 & \text{ if } q \equiv -1 \bmod 3 \\
0 & \text{ if } 3 \mid q.
\end{cases}\]

\begin{theorem} \label{totalnumberoforbits} \cite[Theorem 2.2]{NSet}
The total number of orbits is
\[N_{orbit} = 
\begin{cases} 
2q+2 + \binom{-3}{q} & \text{ if } p \geq 5\\
2q+2 & \text{ if } p = 3 \text{ and } q \equiv 1 \bmod 4\\
2q+1 & \text{ if } p =3 \text{ and } q \equiv 3 \bmod 4\\
2q+\binom{-3}{q} & \text{ if } p = 2.
\end{cases}\]
\end{theorem}

\subsection{The splitting behavior when $r=4$}

We name the cases as follows.

\begin{center}
\begin{tabular}{lll} \hline
$S$ & Description & Notation \\ \hline
Split & four points of degree $1$ & $(\cdot, \cdot, \cdot, \cdot)$ \\
Split/Quad & two points of degree $1$ and one of degree $2$ & $(\cdot, \cdot, (\cdot -\cdot))$ \\
Quad & two points of degree $2$ & $((\cdot - \cdot), (\cdot -\cdot))$\\ 
Cubic & one point of degree $1$ and one of degree $3$ & $(\cdot, (\cdot - \cdot -\cdot))$ \\
Quartic & one point of degree $4$ & $((\cdot - \cdot - \cdot -\cdot))$ \\ \hline
\end{tabular}
\end{center}

Given $S$, let $N_{orbit}(S)$ denote the number 
of orbits of $4$-sets of $\mathbb{P}^1(\overline{k})$ with splitting behavior $S$ under ${\rm PGL}_2(k)$.
 
\subsection{The case $p=2$}

\begin{prop} \cite[Proposition 3]{Nart3} 
If $p = 2$, 
then $N_{orbit}(S)$ is given by:
\begin{center}
\begin{tabular}{|c |c |c |c | c|c |}
\hline
$S$ & Split & Split/Quad & Quad & Cubic & Quartic \\ \hline 
$N_{orbit}(S)$ & $(q+2\binom{-3}{q})/6$ & $q/2$ & $(q-2)/2$ & $(q+3 + 2 \binom{-3}{q})/3$ & $q/2$ \\ \hline 
\end{tabular}
\end{center}
\end{prop}

Note that these sum to the quantity $N_{orbit}$ when $p=2$ in 
Theorem~\ref{totalnumberoforbits}.

\subsection{Number of split orbits}

Let $N_{orbit}(split)$ be the number of split orbits, meaning 
the number of orbits of $4$ distinct points of $\mathbb{P}^1(k)$ under the action of 
${\rm PGL}_2(k)$.

\begin{theorem} \label{Goppa} \cite[Theorem C]{Goppa}
If $p$ is odd, then $N_{orbit}(split) = (q+3 + 2\binom{-3}{q})/6$.
\end{theorem}

\subsection{The number of orbits in the non-split cases when $p$ is odd}

We now determine the number of orbits $N_{orbit}(S)$ when $p$ is odd and $S$ is not split.

Recall that $\binom{-1}{q}$ 
equals $1$ if $q \equiv 1 \bmod 4$ and equals $-1$ if $q \equiv 3 \bmod 4$.

\begin{prop} \label{Pnorb}
Let $p$ be odd.
For each non-split $S$, the number $N_{orbit}(S)$ of orbits of $4$-sets of $\mathbb{P}^1(\overline{k})$ with splitting 
behavior $S$ under ${\rm PGL}_2(k)$ is as follows:
\begin{enumerate}
\item ($S$ split/quadratic) $N_{orbit}(S) = (q+1)/2$.
\item ($S$ quadratic) $N_{orbit}(S) = (q-1)/2$.
\item ($S$ cubic) $N_{orbit}(S) = (q+3+2 \binom{-3}{q})/3$.
\item ($S$ quartic) $N_{orbit}(S)=(q+\binom{-1}{q})/2$ if $p=3$ 
and $N_{orbit}(S) =(q+1)/2$ if $p \geq 5$.
\end{enumerate}
\end{prop}

Note that $N_{orbit}$ equals the sum of $N_{orbit}(S)$ over the different splitting behaviors $S$ (including the split one from Theorem~\ref{Goppa}).

The proof of Proposition~\ref{Pnorb} is contained in the following subsections.
Recall Burnside's Lemma: 
if $\Gamma$ is a group acting on a set $I$, then the number of orbits of $I$ under $\Gamma$ is $$|I /\Gamma| = |\Gamma|^{-1}\sum_{\gamma \in \Gamma} |{\rm Fix}_{\gamma}|.$$

\subsection{Split/quadratic case}

\begin{lemma}  \label{splitquad}
If $p$ is odd and $S$ is split/quadratic then $N_{orbit}(S) = (q+1)/2$.
\end{lemma}

\begin{proof}
Let $I$ be the set of irreducible monic degree $2$ polynomials $f(x) \in k[x]$.
If $\theta, \theta' \in {\mathbb F}_{q^2}\backslash {\mathbb F}_q$ are the roots of some 
$f(x) \in I$, then $W =  \{0, \infty, \theta, \theta'\}$ is a split/quadratic $4$-set.
Conversely, the orbit of every split/quadratic 4-set contains a representative of the form
$W = \{\infty, 0, \theta, \theta'\}$.
So it suffices to determine the number of different orbits among $4$-sets of the form $W$.
For this, it suffices to determine the number of orbits of $f(x) \in I$ under $\Gamma_W$.

If $\gamma \in \Gamma_W$, then either $\gamma$ fixes $0$ and $\infty$ or 
$\gamma$ transposes $0$ and $\infty$.  So $\Gamma_W = \Gamma_0 \cup \Gamma_\infty$ where
$\Gamma_0 = \{\gamma_a: x \mapsto ax \mid a \in k^*\}$ and 
$\Gamma_\infty = \{\bar{\gamma}_b: x \mapsto b/x \mid b \in k^*\}$.

The action of $\Gamma_W$ on $I$ is given as follows.
If $\gamma_a \in \Gamma_0$, then $\gamma_a(f(x)) = f(ax)/a^2$.
To define the action of $\bar{\gamma}_b \in \Gamma_\infty$ on $I$, we think of
$I$ as the set of irreducible (possibly not monic) degree $2$ polynomials $f(x) \in k[x]$, 
up to equivalence $\sim$, where $f_1(x) \sim f_2(x)$ if and only if $f_1(x)=\lambda f_2(x)$ for some $\lambda \in k^*$.
Then $\bar{\gamma}_b(f(x))= x^2 f(\frac{b}{x})$.
Writing $f(x)=A_0 x^2 + A x + B$, then $\bar{\gamma}_b(f(x))= A_0 b^2 + A b x + Bx^2$.

\bigskip

{\bf Claim 1:} 
Let $I_0 = \{f(x) = x^2 + B \mid -B \in k^* \text{ is a quadratic non-residue}\}$. 
Then $I_0 \subset I$ and $\#I_0=(q-1)/2$. 
The elements of $I_0$ are in one orbit under the action of $\Gamma_W$.

Proof of Claim 1: It suffices to show that the elements of $I_0$ 
are in one orbit under the action of $\Gamma_0$.
If $f(x) \in I_0$ is in $\textup{Fix}_{\gamma_a}$ then 
		\begin{eqnarray*}
		x^2+B = f(x) = \gamma_a(f(x)) = f(a x)/a^2 = x^2 + B/a^2.
		\end{eqnarray*}
		Thus, $\textup{Fix}_{\gamma_a}=\emptyset$ unless $a = \pm 1$ and 
		$\textup{Fix}_{\gamma_{1}} = \textup{Fix}_{\gamma_{-1}}=I_0$. 
		By Burnside's lemma,
		 $|I_0 /\Gamma_0| = (q-1)^{-1}2(q-1)/2 = 1$.

\medskip
 		
{\bf Claim 2:} 
Let $I_{\not = 0} = I - I_0$. 
Then $\# I_{\not = 0} = (q-1)^2/2$.  
The elements of $I_{\not = 0}$ form $(q-1)/2$ orbits under the action of $\Gamma_W$.

\medskip

Proof of Claim 2:
Note that $\# I_{\not = 0} = \#I - \#I_0 = (q^2-q)/2-(q-1)/2 = (q-1)^2/2$. 
We apply Burnside's Theorem to the 
action of $\Gamma_W$ on $I_{\not = 0}$.
If $f(x) \in \textup{Fix}_{\gamma_a}$, then 
	\[x^2+Ax +B = f(x) = \gamma_a(f(x)) = f(a x)/a^2 = x^2 + Ax/a + B/a^2.\]
	So $\textup{Fix}_{\gamma_a}$ is trivial unless $a = 1$ and 
	$\textup{Fix}_{\gamma_1} =  I_{\not = 0}$.
	If $f(x) \in \textup{Fix}_{\bar{\gamma}_b}$, then 
		\begin{eqnarray*}
		A_0 x^2 + Ax + B \sim x^2 f(\frac{b}{x}) =  A_0 b^2 + Ab x + Bx^2,
		\end{eqnarray*}
which is true if and only if $b=B/A_0$. 
So each $f(x) \in I_{\not = 0}$ is fixed by exactly one $\bar{\gamma}_b$.  
	Thus,  \[|I_{\not = 0} /\Gamma_W| = (2(q-1))^{-1}(2 (q-1)^2/2) = (q-1)/2.\]

\end{proof}


\subsection{Quadratic case}

The goal of this section is to prove Lemma~\ref{quadratic}, which states that 
the number of orbits of quadratic $4$-sets of $\mathbb{P}^1(\overline{k})$ 
under ${\rm PGL}_2(k)$ is $(q-1)/2$ when $p$ is odd.

Fix a quadratic non-residue $s \in k$.  Let $f_1(x)=x^2-s$. 
The orbit of every quadratic $4$-set contains a $4$-set of the form $W_s = \{\pm \sqrt{s}, \tau, \tau'\}$ where $\{\tau,\tau'\}$ are the roots of an irreducible monic degree $2$ polynomial $f_2(x)$ 
that is not equal to $f_1(x)$. 

We investigate the action of $\Gamma_{W_s}$ on the set of such polynomials $f_2(x)$.

\begin{lemma} \label{Lquadfirst}
Let $W_s=\{\pm \sqrt{s}\}$, then $|\Gamma_{W_s}|=2(q+1)$.
The elements of $\Gamma_{W_s}$ are represented by the equivalence classes in ${\rm PGL}_2(k)$ of 
these matrices (for $a, c \in k$ not both equal to 0):
\begin{equation} 
\label{Egamma12}
\gamma_1(a,c) = \left( \begin{array}{ll}a & cs \\ c& a\end{array}\right)  {\text and \ }
\gamma_2(a,c) = \left( \begin{array}{ll}a & -cs \\ c& -a\end{array}\right).
\end{equation}
\end{lemma}

\begin{proof}
If $\gamma \in \Gamma_{W_s}$, then $ \gamma (f_1(x)) \sim f_1(x)$. 
This implies that $x^2-s$ is equivalent to:
\begin{eqnarray*}
(cx+d)^2 f_1\left(\frac{ax+b}{cx+d}\right) &=& (a^2x^2+2abx+b^2) - s(c^2x^2+2cdx+d^2)\\
&=& x^2(a^2-sc^2)+x(2ab-2cds)+b^2-sd^2.
\end{eqnarray*}
So $ab = cds$ and $(b^2-sd^2)/(a^2-sc^2) = -s$.

If $a = 0$, then $cds = 0$ which implies $d = 0$ since $c \neq 0$.  Then $b = \pm cs$.

If $a \neq 0$, then $b = cds/a$ and one computes that $s(a^2-sc^2)(a^2-d^2) = 0$.
Since $s$ is a quadratic non-residue, this implies that $a = \pm d$.

Thus $\gamma$ is either $\gamma_1(a,c)$ or $\gamma_2(a,c)$ for some pair $(a,c) \in k^2$, with $a,c$ not both $0$.  There are $2(q^2-1)/(q-1)$ equivalence classes of $\gamma$ in 
${\rm PGL}_2(k)$.  Representatives of these are $\gamma_1(a,c)$ and $\gamma_2(a,c)$ where 
either $(a,c)=(0,1)$ or $a =1$ with $c \in k$.
\end{proof}

Let $\Gamma_1$ be the subgroup of ${\rm PGL}_2(k)$ of elements $\gamma$ that fix 
each of $\sqrt{s}$ and $-\sqrt{s}$.  Note that $\gamma \in \Gamma_1$ if and only if 
$\gamma = \gamma_1(a,c)$ for some pair $(a,c)$ not both zero.

Let $\Gamma_2$ be the non-trivial coset of $\Gamma_1$ in $\Gamma_{W_s}$.
Note that every $\gamma \in \Gamma_2$ exchanges $\sqrt{s}$ and $-\sqrt{s}$. 
Also $\gamma \in \Gamma_2$ if and only if $\gamma =  \gamma_2(a,c)$ for some pair $(a,c)$ not both zero.

\begin{lemma} Then $\Gamma_1 \simeq C_{q+1}$
and $\Gamma_{W_s} \simeq D_{q+1}$, the dihedral group of order $2(q+1)$.	
\end{lemma}		

\begin{proof}		
Consider the subgroup $\tilde{\Gamma}_{0, \infty}$ of elements of ${\rm PGL}_2(\FF_{q^2})$ that fix each of $0$ and $\infty$; it is cyclic of order $q^2-1$.		
The subgroup $\tilde{\Gamma}_1$ of elements of ${\rm PGL}_2(\FF_{q^2})$ that fix each of $\sqrt{s}$ and $-\sqrt{s}$ is conjugate to $\tilde{\Gamma}_{0, \infty}$,	
and thus is also cyclic of order $q^2-1$.  		
Then $\Gamma_1 = \tilde{\Gamma}_1 \cap {\rm PGL}_2(\FF_{q})$ which must also be cyclic.	
By Lemma \ref{Lquadfirst},
$|\Gamma_1|=q+1$.

The non-identity coset $\Gamma_2$ of $\Gamma_1$ in $\Gamma_{W_s}$
is represented by $\bar {\gamma}_2:x \mapsto -x$.
The conjugation action of $\bar \gamma_2$ on $\Gamma_1$ is inversion:
$\bar \gamma_2 \gamma_1 \bar \gamma_2^{-1} = \gamma_1^{-1}$. 
Thus $\Gamma_{W_s} \simeq D_{q+1}$.
\end{proof}

\begin{lemma} \label{Lfixgamma1}
Let $\gamma_1=\gamma_1(a,c) \in \Gamma_1$ as in \eqref{Egamma12}.
Let $f_2(x)=x^2+Ax+B$ be an irreducible monic degree $2$ polynomial in $k[x]$ 
with $f_2(x) \not = f_1(x)$.
Then:
\begin{enumerate}
\item \begin{equation} \label{quadgeneralgamma1}
\gamma_1(f_2(x)) =  x^2+
x \frac{A(a^2+c^2s) + 2ac(B+s)}{a^2+Aac+Bc^2} + \frac{c^2s^2+Aacs+Ba^2}{a^2+Aac+Bc^2}.
\end{equation}
\item If $\gamma_1 \not = {\rm Id}$, 
then $f_2(x) \in {\rm Fix}_{\gamma_1}$ if and only if $a=0$ and $B=s$.
\end{enumerate}
\end{lemma}

\begin{proof}
\begin{enumerate}
\item The polynomial on the right of \eqref{quadgeneralgamma1} equals $(cx+a)^2 f_2((ax+cs)/(cx+a))$.
\item Suppose $f_2(x) \in {\rm Fix}_{\gamma_1}$.
Then $f_2(x)=\gamma_1(f_2(x))$.
By part(1), this implies
\begin{eqnarray*}
A^2ac+ABc^2-2acs-Ac^2s-2Bac &=& 0,\\
ABac+B^2c^2 - c^2s^2-Aacs &=& 0.
\end{eqnarray*}

Since $\gamma_1 \not = {\rm Id}$, we can suppose $c \not = 0$.
If $a=0$, without loss of generality, take $c=1$; then $A(B-s) = 0$ and $B^2-s^2 = 0$, so $B= \pm s$.
If $B = -s$, then $A = 0$, which contradicts the fact $f_2(x) \not = f_1(x)$.  Thus $B = s$. 

If $ac \not = 0$; without loss of generality, take $a=1$.  
We will show that the hypothesis $f_2(x) \in \textup{Fix}_{\gamma_1}$
leads to a contradiction. 
Dividing the equations above by $c$ yields
		\begin{eqnarray*}
		Ac(B-s)-2(B+s)+A^2 &=& 0,\\
		c(B^2-s^2)+A(B-s) &=& 0.
		\end{eqnarray*}
If $B \neq s$, some arithmetic shows $A = -c(B+s)$ and then $(B+s)(c^2s-1) = 0$.  Thus $B=-s$, which contradicts the fact
$f_2(x) \not = f_1(x)$.
If $B = s$, some arithmetic shows that $A^2-4B = 0$, which contradicts the fact that $f_2(x)$ is irreducible.
\end{enumerate}
\end{proof}

\begin{lemma} \label{quadratic}
If $p$ is odd and $S$ is quadratic, then $N_{orbit}(S)=(q-1)/2$.
\end{lemma}

\begin{proof}
Every quadratic orbit contains a $4$-set of the form $W = \{\pm \sqrt{s}, \tau, \tau'\}$ where $\{\tau,\tau'\}$ are the roots of an irreducible monic quadratic polynomial $f_2(x)$ (not equal to $f_1(x)$). 
Let $$I = \{f_2(x) = x^2+Ax+B \mid A,B \in k, \  f_2(x) \text{ irreducible}, \ f_2(x) \neq x^2-s\}.$$
Then $\#I =  (q^2-q)/2-1 = (q+1)(q-2)/2$.

\bigskip

{\bf Claim 3:}  The number of orbits of $I$ under $\Gamma_1$ is $(q-1)/2$.

The reason for this is that $|\textup{Fix}_{\rm Id}| = \#I$ and 
$\textup{Fix}_{\gamma_1}$ has size $(q+1)/2$ if $(a,c)=(0,1)$ and is trivial 
otherwise by Lemma~\ref{Lfixgamma1}.
By Burnside's Lemma,
\[|I /\Gamma_1 | = (q+1)^{-1}\left((q+1)(q-2)/2)+(q+1)/2\right) = (q-1)/2.\]

To finish the proof of Lemma \ref{quadratic}, note that $\Gamma_2$ is a coset of $\Gamma_1$ in $\Gamma_{W_s}$.  A representative of this coset is the 
map $\bar{\gamma}_2(x)=-x$.

{\bf Claim 4:} 
Given $f_2(x) = x^2+Ax+B \in I$, there exists 
$\gamma_1=\gamma_1(a,c) \in \Gamma_1$ such that 
$\gamma_1 \cdot f_2(x) = \bar{\gamma}_2 \cdot f_2(x)$.

To see this, note that $\bar{\gamma}_2 \cdot f_2(x) = f_2(-x) = x^2-Ax+B$.
If $A=0$, the claim is true, taking $\gamma_1 = {\rm Id}$.
Let $A \not = 0$. 
By \eqref{quadgeneralgamma1}, the condition $\gamma_1 \cdot f_2(x) = x^2-Ax+B$ 
is equivalent to
\begin{equation} \label{Equadorbit}
B = \frac{c^2s^2+Aacs+Ba^2}{a^2+Aac+Bc^2}, \ {\rm and} \ -A = \frac{A(a^2+c^2s) + 2ac(B+s)}{a^2+Aac+Bc^2}.
\end{equation}
If $a^2+Aac+Bc^2 \not = 0$, this is equivalent to
\[0  =  c^2(B^2-s^2) + Aac(B-s), \ {\rm and} \ 0 = a^2(2A) + a(2cs+2Bc+A^2c)+Ac^2(B+s).\]
The two equations are satisfied when $c=1$ and $a=-(B+s)/A$.

	Claim 4 implies that the number of orbits under $\Gamma_{W_s}$ is the same as the number of orbits under $\Gamma_1$, namely $(q-1)/2$.
\end{proof}


\subsection{Cubic case}

\begin{lemma} \label{cubic}
If $p$ is odd and $S$ is cubic, then $N_{orbit}(S)$ equals
$(q+5)/3$ if $q \equiv 1 \bmod 3$ and equals $(q+1)/3$ if $q \equiv -1 \bmod 3$.
\end{lemma}

\begin{proof} 
Let $I$ be the set of irreducible monic degree $3$ polynomials $f(x) \in k[x]$.
If $\theta, \theta', \theta'' \in {\mathbb F}_{q^3}\backslash {\mathbb F}_q$ are the roots of some $f(x) \in I$, then $W =\{\infty, \theta, \theta', \theta''\}$ is a cubic $4$-set.
Conversely, the orbit of every cubic 4-set contains a representative of the form
$W=\{\infty, \theta, \theta', \theta''\}$.

It suffices to find the number of orbits among $4$-sets of the form $W$.
For this, it suffices to find the number of orbits of $f(x) \in I$ under $\Gamma_W$.
Here $\Gamma_W =\{\gamma: x \rightarrow ax+b, \ a \in k^*, b \in k\}$ and $|\Gamma_W| = q(q-1)$.

If $f(x) \in I$, write $f(x) = x^3+Ax^2+Bx+C$.
Then $f(x) \in \textup{Fix}_\gamma$ if and only if $f(ax+b)/a^3=f(x)$, which is equivalent to
\[A = (3b+A)/a, \ B = (3b^2+2Ab+B)/a^2, \ C =(b^3+Ab^2+Bb+C)/a^3,\]
or
\begin{equation}\label{ECsimplify}
(a-1)A = 3b, \ (a^2-1) B = 3b^2+2Ab, \ (a^3-1)C = b^3+Ab^2+Bb.
\end{equation}


Case 1: If $a=1$, then $b=0$ and $\gamma={\rm Id}$, with $|\textup{Fix}_{\rm Id}| = |I| = (q^3-q)/3$.

Case 2: If $a=-1$, then $\gamma(x) = -x+b$ has order $2$, which cannot stabilize a cubic point, so $\textup{Fix}_{\gamma} = \emptyset$.

Case 3: 
If $a^3 = 1$ with $a \not = 1$, then $f(x) \in \textup{Fix}_{\gamma}$ if and only if
\[A = 3b/(a-1) \ {\rm and } \ B = 3b^2/(a-1)^2.\]
Let $x = y-A/3$ then $f(y-A/3) = y^3 + C - \left(\frac{b}{a-1}\right)^3$,
which is irreducible if and only if $(\frac{b}{a-1})^3 - C$ is not a cube in $k^*$.
If $q \equiv -1 \bmod 3$, this is impossible so $\textup{Fix}_{\gamma} = \emptyset$.
If $q \equiv 1 \bmod 3$, there are $2(q-1)/3$ non-cubes and 
$2q$ choices of $\gamma(x)=ax+b$ for each of these, 
which contributes in total $4q(q-1)/3$.
	
Case 4: If $a^3 \neq 1, a \neq -1$, we compute that
	\[A=3b/(a-1), \ B = 3b^2/(a-1)^2, \ C=b^3/(a-1)^3.\]
Let $x=y-A/3$.  Then $f(y-A/3)=y^3$, which is reducible. 
So this case does not occur. 

		Hence, the total number of orbits is
	$$|I /\Gamma_W| = \left\{ 
			\begin{array}{lr}
			(q(q-1))^{-1}\left((q^3-q)/3 + 4q(q-1)/3\right) = (q+5)/3 &{\rm if } \ q \equiv 1 \bmod 3\\
			(q(q-1))^{-1}\left((q^3-q)/3 \right) = (q+1)/3 &
			{\rm if } \ q \equiv -1 \bmod 3.
			\end{array}
			\right.$$
			
\end{proof}

Now consider the cubic case when $k$ has characteristic $3$.
The following fact about $f(x) = x^3+Bx+C \in k[x]$ will be useful:
$f(x)$ is irreducible in $k[x]$ if and only 
if $-B=w^2$ for some $w \in k$ and $\text{tr}(C/w^3) \neq 0$ \cite[Theorem 2]{Williams}.

\begin{lemma} \label{cubic}
If $p=3$ and $S$ is cubic, then $N_{orbit}(S)=(q+3)/3$.
\end{lemma}

\begin{proof} 
From \eqref{ECsimplify} in the previous proof, 
$f(x) \in \textup{Fix}_{\gamma}$ if and only if
\[(a-1)A = 0, \ (a^2-1)B = 2Ab, \ (a^3-1)C =b^3+Ab^2+Bb.\]

Suppose $f(x) \in \textup{Fix}_{\gamma}$ with $\gamma(x)=ax+b$.

If $a \not = \pm 1$, then $A=B=0$.  
Since $f(x)=x^3+C$ factors in $k[x]$, this implies $\textup{Fix}_{\gamma} = \emptyset$.

If $a = 1$ and $b=0$, then $|\textup{Fix}_{\rm Id}|= (q^3-q)/3$.

If $a = 1$ and $b \neq 0$, then $A = 0$ and $B=-b^2$.
So $f(x) = x^3-b^2x+C$, which is irreducible if and only if $\text{tr}(C/b^3) \neq 0$.
Let $\kappa = {\rm Fr}^{-1}(C)$; then $C=\kappa^3$.
Then $\text{tr}(C/b^3) = \text{tr}({\rm Fr}(\kappa/b)) =  \text{tr}(\kappa/b)$.
Recall that $\text{tr}:{\mathbb F}_q \to {\mathbb F}_3$ is a surjective group 
homomorphism.
So, for each $b \in {\mathbb F}_q^*$, the number of 
$C \in {\mathbb F}_q$ such that $f(x)$ is irreducible is $2q/3$
and $|\textup{Fix}_{\gamma}| = 2q/3$.

If $a = -1$, then $A=0$ and $C=b(b^2+B)$.  If $-B=w^2$ for some $w \in k^*$, let 
$z=b/w$.  
Then $\text{tr}(b(b^2-w^2)/w^3) = \text{tr}(z(z-1)(z-2))=0$ so $f(x)$ factors in $k[x]$.
This implies $\textup{Fix}_{\gamma} = \emptyset$.

Thus 
\[|I /\Gamma| = (q(q-1))^{-1} ((q^3-q)/3 + (q-1)2q/3) = (q+3)/3.\]
\end{proof}


\subsection{Quartic case}

\begin{lemma} (Quartic Case)\label{quartic}
If $S$ is quartic, then $N_{orbit}(S)$ equals $(q-1)/2$ if $p=3$ and $q \equiv -1 \bmod 4$ and 
equals $(q+1)/2$ otherwise.
\end{lemma}

\begin{proof}
The total number of orbits is found in Theorem~\ref{totalnumberoforbits}
(\cite[Theorem 2.2]{NSet}).  From this, 
we subtract the number of orbits from the non-quartic splitting behaviors found in Theorem~\ref{Goppa} (\cite[Theorem C]{Goppa}), and Lemmas~\ref{splitquad}, \ref{quadratic}, and \ref{cubic}.
\end{proof}

\section{The mass formula for $4$ branch points} \label{Sweightsum4}

In this section, we consider the mass formula for Artin-Schreier curves with automorphism
when the number of branch points of the cover is $4$.
For a ramification data $R$ of length $r=4$, we provide formulas for 
\[Z_{R,S}(q):=\sum_{[C], \text{ type } R,S/\simeq k} |\cent|^{-1},\]
when $S$ is split in Section~\ref{Ssplit}, 
when $S$ is non-split and $p=2$ in Section~\ref{Sp=2}, 
and some cases when $S$ is non-split and $p$ is odd in Section~\ref{Snonsplit}. 

From a moduli-theoretic perspective, it is natural to fix the ramification data $R$ and 
allow the splitting behavior $S$ to vary.
By combining the results in this section, we obtain the following conclusions about
\[Z_R(q):=\sum_{[C] \text{ type } R/\simeq k} |\cent|^{-1}.\]

\subsection{Conclusions} \label{Sconclusions}

Consider ramification data
$R = \{\epsilon_1, \epsilon_2, \epsilon_3, \epsilon_4\}$ with $\epsilon_i \not \equiv 1 \bmod p$.
Let
$e_i = \epsilon_i-1$ and define \[E := E(R) = \sum_{i=1}^4 \left(e_i - 1 - \left\lfloor \frac{e_i}{p} \right\rfloor\right).\]

\begin{cor}
The mass formula $Z_R(q)$ has the following formulas:
\begin{enumerate}
\item If $\epsilon_1, \epsilon_2, \epsilon_3, \epsilon_4$ are distinct, 
then $Z_{R}(q) =  (q-1)^4q^{E(R)} T$ where
$T=(q+3 + 2\binom{-3}{q})/6$ if $p \geq 5$ and $T=(q+ 2\binom{-3}{q})/6$ if $p=2$.
\item If $\epsilon_1, \epsilon_2, \epsilon_3$ are distinct and $\epsilon_3= \epsilon_4$, 
then $Z_R(q) = (q-1)^3q^{E(R)} T'$,
where $T'=(1/3)(2q^2+ q + (q-1) \binom{-3}{q})$.

\item If $\epsilon_1 = \epsilon_2 = \epsilon_3$ and $\epsilon_4 \not = \epsilon_3$, 
then $Z_R(q) = (q-1)^2q^{E(R)} (q^3+q+ \gamma)$,
where $\gamma=(2/3)(q^2+q+1)$ if $p=3$ and $\gamma= (-2/3)(q^2-2q+1)$ if $p \geq 5$ and $q \equiv 1 \bmod 6$ and $\gamma=0$ otherwise.
\end{enumerate}
\end{cor}

\begin{proof}
\begin{enumerate}
\item 
Since all pole orders are distinct, then $S$ is split. 
The result follows from Case~1 of Proposition~\ref{Results4S}.
\item 
The condition on the pole orders implies that $S$ is either split or split/quadratic.
The result follows by adding Case 2 of Proposition~\ref{Results4S} and the case of having different pole orders from Propositions~\ref{Results4NS} and \ref{Results4NSodd}.
\item The condition on the pole orders implies that $S$ is either split, split/quadratic with different pole orders, or cubic.
The result follows by adding Case 3 of Proposition~\ref{Results4S} and the appropriate cases from Propositions~\ref{Results4NS} and \ref{Results4NSodd}.
\end{enumerate}
\end{proof}

\subsection{The split case with $r=4$ points} \label{Ssplit}

\begin{prop} \label{Results4S}
When $S$ is split, then $Z_{R,S}(q) = (q-1)^4q^{E(R)} T$,
where $T$ is as follows:

\begin{center}
\begin{tabular}{| l | l | l |} \hline
Case & Condition on $R$ & $T$\\ \hline
1 & $\epsilon_1, \epsilon_2, \epsilon_3, \epsilon_4$ distinct 
&   $(q+3 + 2 \binom{-3}{q})/6$, $p$ odd; $(q+2\binom{-3}{q})/6$, \ $p=2$\\ 
2 & $\epsilon_1, \epsilon_2, \epsilon_3$ distinct, $\epsilon_3 = \epsilon_4$ 
 & $(q+2\binom{-3}{q})/6$ \\ 
3 & $\epsilon_1 = \epsilon_2 = \epsilon_3 \not = \epsilon_4$ 
 & $(q-2)/6$ \\ 
4 & $\epsilon_1 = \epsilon_2 \not = \epsilon_3 = \epsilon_4$ 
 &  $(q+2\binom{-3}{q})/12$\\ 
5 & $\epsilon_1 = \epsilon_2 = \epsilon_3 = \epsilon_4$ & $(q-2)/24$  \\
\hline
\end{tabular}
\end{center}
\end{prop}

\begin{proof}
The action of ${\rm PGL}_2(k)$ on ${\mathbb P}^1(k)$ is triply transitive. 
So each split orbit contains a representative of the form 
$B_t=\{0,1,\infty, t\}$ for some $t \in k-\{0,1\}$. 
Let $\Gamma_t = {\rm Stab}(B_t) \subset S_4$.
Recall the information about $\Gamma_t$ from Lemma~\ref{C2xC2}.

Let $W_\circ$ be a weighted branch divisor with ramification data $R$ supported on $B_t$.
By Theorem \ref{weightedsum}, $Z_{R,S}(q) = p^{-1} |N_{W_\circ}| \cdot T$ where 
$T=\sum_{H}\frac{|\theta_H|}{|H|}$.  
By Proposition~\ref{PcountN}, $|N_{W_\circ}| = p(q-1)^4q^{E(R)}$.

Let $\Gamma_{W_\circ} = {\rm Stab}(W_\circ) \subset {\rm PGL}_2(k)$.
This subgroup depends on the values in $R$ because 
if the pole orders of two branch points are different, 
then no automorphism in $\Gamma_{W_\circ}$ can map one of them to the other.
For each representative $H$ of a conjugacy class in $S_4$, we consider the set $\theta_H$ of orbits of $t$ for which 
$\Gamma_{W_\circ}$ is conjugate to $H$.
We need to compute $T$.

The proof is then a case-by-case analysis using Lemmas~\ref{Lspecialorbit} and \ref{C2xC2} and Theorem~\ref{Goppa}.

Case 1: then $\theta_H$ is empty unless $H=\{\rm id\}$, 
so $T =N_{orbit}(split)$. 

Case 2: Without loss of generality, the two branch points with pole order 
$\epsilon_3-1$ are $0$ and $\infty$.  Then $|H_t|=1$ unless $t=-1$.
If $p=2$, then this does not happen and $T=N_{orbit}(split)$.
If $p$ is odd, then there is one orbit where $|H_t|=2$, so $T=N_{orbit}(split) - 1 + 1/2$.

Case 3: 
Set the three branch points with pole order $\epsilon_1-1$ to be $0, 1,\infty$. 

If $p =2$, then $\zeta_3, \zeta_3^2 \in k$ if and only if $q \equiv 1 \bmod 3$;
this orbit has $|H_t|=3$, so $T=(1/3) + (q+2)/6 -1$ if $q \equiv 1 \bmod 3$ and 
$T= (q-2)/6$ if $q \equiv -1 \bmod 3$.

If $p=3$, then the orbit of $-1$ has $|H_t|=6$ so $T=(1/6) + (q+3)/6 -1$.

If $p=5$, then the orbit of $-1$ has $|H_t|=2$ and $\zeta_6 \in k$ 
if and only if $q \equiv 1 \bmod 3$, in which case it has $|H_t|=3$;
so $T= ((q+5)/6 - 2) +(1/2) + (1/3) = (q-2)/6$ if $q \equiv 1 \bmod 3$ and 
$T=((q+1)/6 - 1) +(1/2) = (q-2)/6$ if $q \equiv -1 \bmod 3$.

Case 4: Set the two branch points with pole order $\epsilon_1-1$ to be $0, \infty$ and 
the two branch points with pole order $\epsilon_3-1$ to be $1, t$.
Then $\gamma_{1,t} = (0, \infty)(1,t) \in H_t$ for all $t$.
The only other elements of $H_t$ are automorphisms $x \mapsto 1/x$ and $x \mapsto -x$ when $t=-1$.
So when $p=2$, then $|H_t|=2$ for all orbits. 
When $p$ is odd then 
$|H_t|=2$ for all but one orbit, for which $|H_t|=4$.
So $T=(N_{orbit}(split) -1)/2 + (1/4)$.

Case 5: then $|H_t| = \Gamma_t$ is as in Lemma \ref{C2xC2}.
So $|H_t| =4$ except in the special cases.

When $p=2$, $q \equiv 1 \bmod 3$, then the orbit of $\zeta_3$ has $|H_t|=12$; so 
$T=(q-2)/24$ when $q \equiv -1 \bmod 3$ and $T=((q+2)/6 -1)/4 + (1/12)$ when $q \equiv 1 \bmod 3$.

When $p=3$, then the orbit of $-1$ has $|H_t|=24$; so $T=((q+3)/6 -1)/4 + (1/24)$.

When $p \geq 5$, then the orbit of $-1$ has $|H_t|=8$ and when $q \equiv 1 \bmod 3$ then the orbit of 
$\zeta_6$ has $|H_t| = 12$.
So $T= ((q+1)/6 -1)/4 + (1/8)$ when $q \equiv -1 \bmod 3$ and 
$T=((q+5)/6 -2)/4 + (1/8) + (1/12)$ when $q \equiv 1 \bmod 3$.
\end{proof}

\subsection{The case when $p=2$} \label{Sp=2}

In the next result, we generalize the case $R=\{2,2,2,2\}$ when $p=2$ found in \cite[Section 3.1]{Nart3}.

\begin{prop} \label{Results4NS}
Let $p=2$.
Let $R=\{\epsilon_1, \epsilon_2,\epsilon_3, \epsilon_4\}$ with $\epsilon_i$ even.
If $S$ is non-split, then $Z_{R,S}(q)$ is given by:
\begin{center}
\begin{tabular}{ll}
\hline
$S$ & $Z_{R,S}(q)$\\
\hline
split/quad if split points have different pole orders & $(q-1)^2(q^2-1)q^{E+1}/2$ \\ 
split/quad if split points have the same pole order & $(q-1)^2(q^2-1)q^{E+1}/4$\\
quadratic & $(q^2-1)^2q^E(q-2)/8$\\
cubic & $(q-1)(q^3-1)q^E(q+1)/3$\\
quartic & $(q^4-1)q^{E+1}/4$\\
\hline\\
\end{tabular}
\end{center}
\end{prop}

As a check, when $R=\{2,2,2,2\}$, these values sum to
$q^5-q^4$; when combined with the $r \leq 3$ data, the sum is $q^5$ as in Theorem~\ref{Nq5}.

\begin{proof}
By Theorem~\ref{weightedsum}, $Z_{R,S}(q) =   p^{-1} |N_{W_\circ}| \cdot T$ where 
$T=\sum_{H}\frac{|\theta_H|}{|H|}$.  
In this case, $|N_{W_\circ}|$ depends on $S$ as in Proposition~\ref{PcountN}.
The computation of $T$ follows from the following table,
which is mostly contained in \cite[Section 3.1]{Nart3}.
\begin{center}
\begin{tabular}{l l l l l} \hline
Case & $H$ & Number of Orbits\\ \hline
split/quad & \{{\rm id}\} if split points have different pole orders & $q/2$ \\ 
split/quad & $C_2$ if split points have same pole order & $q/2$ \\ 
quadratic & $C_2 \times C_2$ & $q/2-1$\\ 
cubic & $\mu_3$ if $t = \infty$ when $q \equiv 1 \bmod 3$ & 1\\
& $\mu_3$ if $t = 0$ when $q \equiv 1 \bmod 3$ & 1\\
& $\{\rm id\}$ & $(q+1)/3$\\ 
quartic& $C_2$ & $q/2$\\ \hline
\end{tabular}
\end{center}
\end{proof}

\subsection{Some non-split cases when $p$ is odd} \label{Snonsplit}

\begin{prop} \label{Results4NSodd}
If $p$ is odd and $S$ is non-split, then $Z_{R,S}(q)$ is given by:
\begin{center}
\begin{tabular}{ll}
\hline
$S$ & $Z_{R,S}(q)$ \\
\hline
split/quad if split points have different pole orders & $(q-1)^2(q^2-1)q^{E+1}/2$ \\
split/quad if split points have same pole orders		& $(q-1)^2(q^2-1)q^{E+1}/4$ \\
cubic if $p \neq 3$ & $(q-1)(q^3-1)(q+1)q^{E}/3$ \\
cubic if $p = 3$ & $(q-1)(q^3-1)(q+3)q^{E}/3$\\
\hline
\end{tabular}
\end{center}
\end{prop}
\begin{proof}
By Theorem~\ref{weightedsum}, $Z_{R,S}(q) =   p^{-1} |N_{W_\circ}| \cdot T$ where 
$T=\sum_{H}\frac{|\theta_H|}{|H|}$.  
In this case, $|N_{W_\circ}|$ depends on $S$ as in Proposition~\ref{PcountN}.

Consider the split/quadratic case.
The orbit of every split/quadratic $4$-set contains a representative 
$W_\circ=\{0,\infty, \theta, \theta'\}$ where $\theta, \theta'$ are the roots of an irreducible polynomial of the form
$f(x) = x^2 + Ax +B \in k[x]$.  Note that $\theta \theta'=B$. 
Then $|N_{W_\circ}| = p(q-1)^2(q^2-1)q^{E}$.  

Consider the cubic case.
The orbit of every cubic $4$-set contains a representative 
$W=\{\infty, \theta, \theta', \theta''\}$ where $\theta, \theta', \theta''$ are the roots of an irreducible polynomial of the form
$f(x) = x^3 + Ax^2 +Bx+C \in k[x]$.  
In this case, $|N_{W_\circ}| = p(q-1)(q^3-1)q^{E}$.  

To determine $T$, for each representative $H$ of a conjugacy class in $S_4$, we find the number of orbits for which 
$\Gamma_W$ is conjugate to $H$. 

\begin{center}
\begin{tabular}{l l l l l} \hline
Case & $H$ & Number of Orbits\\ \hline
split/quad when & $\{\pm x\}$ when $\theta' = -\theta$ & 1\\
split points have different pole orders & $\{x\}$ & $(q-1)/2$\\ \hline

split/quad when & $\{ \pm x, \pm B/x\}$ when $\theta' = -\theta$ & 1\\ 
split points have same pole orders & $\{x,B/x\}$& $(q-1)/2$\\ \hline

cubic when $p \neq 3$& $\{x,\zeta_3 x, \zeta_3^2 x\}$ when $q \equiv 1 \bmod 3$ & 2\\
 & $\{x\}$ when $q \equiv 1 \bmod 3$ & $(q-1)/3$\\
& $\{x\}$  when $q \equiv -1 \bmod 3$ & $(q+1)/3$\\ \hline
cubic when $p = 3$& $\{x\}$ & $(q+3)/3$\\ \hline
\end{tabular}
\end{center}

We give the details of the proof in two cases.

{\bf Case: split/quad when split points have same pole orders.}

By Claim 1 of Lemma~\ref{splitquad}, the set $I_0$ of irreducible polynomials of the form $f(x)=x^2+B$
forms one orbit under the action of $\Gamma_{W_\circ}$.
Then $\Gamma_{W_\circ} = \{\pm x, \pm B/x\}$ and $|\Gamma_{W_\circ}| = 4$.  

By Lemma~\ref{splitquad}, there are $(q-1)/2$ other orbits. 
For these, $\Gamma_{W_\circ}  = \{x, B/x\}$ and $|\Gamma_{W_\circ}| = 2$.  

Thus $T = (1/4) + ((q-1)/2)/2$ and 
$p^{-1} |N_{W_\circ}| \cdot T = (q-1)^2(q^2-1) q^E q/4$.

{\bf Case: cubic when $p \neq 3$ and $q \equiv 1 \bmod 3$.}

There are $2(q-1)/3$ irreducible polynomials of the form $f(x) = x^3+C$.
The group $k^*$ acts on these by $\gamma\cdot f(x) = f(ax)/a^3$ for $a \in k^*$.  This scales $C$
by a cube.  Each is stabilized by $\mu_3 \subset k^*$.
So these polynomials form two orbits, each with $|H| = 3$.  
There are $(q+5)/3-2$ orbits for which $H = \{x\}$ and $|H| = 1$.  
Thus $T=2/3 + (q+5)/3-2 = (q+1)/3$ and $p^{-1} |N_{W_\circ}| \cdot T = (q-1)(q^3-1)q^{E}(q+1)/3$.
\end{proof}

\begin{remark}
Here is a problem for someone interested in finite geometry.  For each orbit in the quadratic and quartic $4$-branch points cases, determine the size of the stabilizers of the orbits. 
Specifically, for each $H \subset S_4$, the goal is to find the number of orbits of the branch locus $W$ for which $\Gamma_W$ is conjugate to $H$. 
With this data, it would be possible to extend the results in this paper to the case of Artin--Schreier covers with genus $g = 6(p-1)/2$.  
\end{remark}



\end{document}